\documentclass[11pt,a4paper]{article}
\usepackage{latexsym,amsfonts,amsmath,amsthm,graphics}


\usepackage{graphicx}       % standard LaTeX graphics tool
\usepackage[latin1]{inputenc}
\usepackage[T1]{fontenc}
\usepackage{microtype} % good font tricks

\usepackage{cite}
\usepackage{array,colortbl}
\usepackage{amsmath,amsfonts,amssymb,bm,amsthm,mathtools,mathrsfs}
\DeclareFontFamily{U}{mathx}{\hyphenchar\font45}
\DeclareFontShape{U}{mathx}{m}{n}{
      <5> <6> <7> <8> <9> <10>
      <10.95> <12> <14.4> <17.28> <20.74> <24.88>
      mathx10
      }{}
\DeclareSymbolFont{mathx}{U}{mathx}{m}{n}
\DeclareFontSubstitution{U}{mathx}{m}{n}
\DeclareMathAccent{\widecheck}{0}{mathx}{"71}

\def\citep#1#2{\cite[{#1}]{#2}}

\usepackage{algorithm}
\usepackage{algorithmic}
\newcommand{\bbF}{{\mathbb{F}}}

\newcommand{\bbN}{{\mathbb{N}}}

\newcommand{\N}{{\mathbb{N}}} % natural numbers {1, 2, ...}
\newcommand{\R}{{\mathbb{R}}} % reals
\newcommand{\Z}{{\mathbb{Z}}} % integers
\newcommand{\FF}{{\mathbb{F}}} % field, finite field
\newcommand{\NN}{{\mathbb{N}}} % natural numbers {1, 2, ...}
\newcommand{\RR}{{\mathbb{R}}} % reals
\newcommand{\ZZ}{{\mathbb{Z}}} % integers
\DeclareSymbolFont{bbold}{U}{bbold}{m}{n}
\DeclareSymbolFontAlphabet{\mathbbold}{bbold}

\newcommand{\bsa}{{\boldsymbol{a}}}

\newcommand{\bsc}{{\boldsymbol{c}}}

\newcommand{\bsg}{{\boldsymbol{g}}}
\newcommand{\bsh}{{\boldsymbol{h}}}

\newcommand{\bsk}{{\boldsymbol{k}}}

\newcommand{\bst}{{\boldsymbol{t}}}

\newcommand{\bsw}{{\boldsymbol{w}}}
\newcommand{\bsx}{{\boldsymbol{x}}}
\newcommand{\bsy}{{\boldsymbol{y}}}
\newcommand{\bsz}{{\boldsymbol{z}}}
\newcommand{\bsA}{{\boldsymbol{A}}}

\newcommand{\bszero}{{\boldsymbol{0}}} % vector of zeros
\newcommand{\bsgamma}{{\boldsymbol{\gamma}}}

\newcommand{\bsDelta}{{\boldsymbol{\Delta}}}

\newcommand{\calD}{{\mathcal{D}}}

\newcommand{\calH}{{\mathcal{H}}}

\newcommand{\calO}{{\mathcal{O}}}

\newcommand{\frakp}{{\mathfrak{p}}}

\newcommand{\fraku}{{\mathfrak{u}}}
\newcommand{\frakv}{{\mathfrak{v}}}

\newcommand{\setu}{{\mathfrak{u}}}
\newcommand{\setv}{{\mathfrak{v}}}

\newcommand{\wal}{{\rm wal}}
\newcommand{\icomp}{\mathrm{i}}
	
\newcommand{\abs}[1]{\left\vert#1\right\vert}

\newcommand{\floor}[1]{\left\lfloor #1 \right\rfloor} % floor
\newcommand{\rd}{\,\mathrm{d}} % differential symbol with tiny space in front for use in integrals
\usepackage[hidelinks,hypertexnames=false,hyperindex=true,pdfpagelabels,linktoc=all]{hyperref}
\usepackage{bookmark}
\makeatletter % to avoid hyperref warnings:
  \providecommand*{\toclevel@author}{999}
  \providecommand*{\toclevel@title}{0}
\makeatother

\usepackage{enumitem}

\theoremstyle{plain}
  \newtheorem{theorem}{Theorem}
  \newtheorem{proposition}{Proposition}
  \newtheorem{lemma}{Lemma}
  \newtheorem{corollary}{Corollary}
\theoremstyle{definition}

\theoremstyle{remark}
  \newtheorem{remark}{Remark}

\usepackage[T1]{fontenc}
\usepackage{lmodern}
\usepackage[a4paper]{geometry}
\usepackage{inputenc}
\usepackage{amsmath,amsfonts,amssymb}
\usepackage{bm}
\usepackage{mathtools}
\usepackage{hyperref}
\usepackage{tikz}
\usepackage{caption}
\usepackage{listings}
\usepackage{float}
\usepackage{pgfplots}
\usepackage{subcaption}
\usepackage{enumitem}
\usepackage{xcolor}
\usepackage{bbm}
\usepackage{booktabs}
\usepackage{multirow}

\allowdisplaybreaks

\usepackage{changebar}

\newcommand{\sob}{{\rm sob}}
\newcommand{\kor}{{\rm kor}}

\begin{document}

	\title{The fast reduced QMC matrix-vector product}

	\author{Josef Dick, 
	Adrian Ebert,
    Lukas Herrmann, 
    Peter Kritzer,
    Marcello Longo 
}

\date{\today}

\maketitle

\begin{abstract}
We study the approximation of integrals of the form $\int_D f(\bsx^\top A) \rd \mu(\bsx)$, where $A$ is a matrix, by quasi-Monte Carlo (QMC) rules $N^{-1} \sum_{k=0}^{N-1} f(\bsx_k^\top A)$. We are interested in cases where the main computational cost in the approximation arises from calculating the products $\bsx_k^\top A$. We design QMC rules for which the computation of $\bsx_k^\top A$, $k = 0, 1, \ldots, N-1$, can be done in a fast way, and for which the approximation error of the QMC rule is similar to the standard QMC error. We do not require that the matrix $A$ has any particular structure.

Problems of this form arise in some important applications in statistics and uncertainty quantification. For instance, this approach can be used when approximating the expected value of some function with a multivariate normal random variable with some given covariance matrix, or when approximating the expected value of the solution of a PDE with random coefficients.

The speed-up of the computation time of our approach is sometimes better and sometimes worse than the fast QMC matrix-vector product from [Josef Dick, Frances Y. Kuo, Quoc T. Le Gia, and Christoph Schwab, Fast QMC Matrix-Vector Multiplication, SIAM J. Sci. Comput. 37 (2015), no. 3, A1436--A1450]. As in that paper, our approach applies to lattice point sets and polynomial lattice point sets, but also applies to digital nets (we are currently not aware of any approach which allows one to apply the fast QMC matrix-vector paper from the aforementioned paper of Dick, Kuo, Le Gia, and Schwab to digital nets).

The method in this paper does not make use of the fast Fourier transform, instead we use repeated values in the quadrature points to derive a significant reduction in the computation time. Such a situation naturally arises from the reduced CBC construction of lattice rules and polynomial lattice rules. The reduced CBC construction has been shown to reduce the computation time for the CBC construction. Here we show that it can additionally be used to also reduce the computation time of the underlying QMC rule. One advantage of the present approach is that it can be combined with random (digital) shifts, whereas this does not apply to the fast QMC matrix-vector product from the earlier paper of Dick, Kuo, Le Gia, and Schwab.

\end{abstract}

\noindent\textbf{Keywords:} Matrix-vector multiplication, quasi-Monte Carlo, high-dimensional integration, 
lattice rules, polynomial lattice rules, digital nets, PDEs with random coefficients.
 \noindent\textbf{2020 MSC:} 65C05, 65D30, 41A55, 11K38.

\section{Introduction and problem setting} \label{sec:intro}

We are interested in approximating integrals of the form
\begin{equation}\label{eq:integral}
\int_{D} f(\bsx^\top A) \rd \mu(\bsx),
\end{equation}
for a domain $D \subseteq \R^s$, an $s \times \tau$-matrix $A \in \R^{s \times \tau}$,
and a function $f: D \to \R$, by quasi-Monte Carlo (QMC) integration rules of the form
\begin{equation} \label{eq:QMC-rule}
	Q_N(f) = \frac1N \sum_{k=0}^{N-1} f(\bsx_k^\top A),
\end{equation}
where we use deterministic cubature points $\bsx_0,\bsx_1,\ldots,\bsx_{N-1} \in D$.
We write $\bsx_k = (x_{1,k}, x_{2,k}, \ldots, x_{s,k})^\top$ for $0\le k \le N-1$. 
In most instances, $D = [0,1]^s$ and the measure $\mu$ is the Lebesgue measure (or $D = \mathbb{R}^s$ and $\mu$ is the measure corresponding to the normal distribution).

Furthermore, define the $N \times s$-matrix
\begin{equation}\label{QMC_matrix}
  X=\begin{pmatrix}
	\bsx_0^\top \\
	\bsx_1^\top  \\
	\vdots \\
	\bsx_{N-1}^\top \\
    \end{pmatrix}
  \in \R^{N \times s},
\end{equation}
whose $N$ rows consist of the different cubature nodes. We are interested in situations where the main computational cost of computing \eqref{eq:QMC-rule} arises from the vector-matrix multiplication $\bsx_k^\top A$ for all $N$ points, i.e., we need to compute $X A$, which requires $\calO(N s \,t)$ operations. Let $A = (\bsA_1, \bsA_2, \ldots, \bsA_{\tau})$, where $\bsA_i \in \R^s$ is the $i$-th column vector of $A$. The main idea is  to construct QMC rules for which the matrix $X$ given in \eqref{QMC_matrix} has some structure such that QMC matrix-vector product $X \bsA_i$ can be computed very efficiently and the integration error of the underlying QMC rule has similar properties as for other QMC rules. Note that our approach works for any matrix $A$ as we do not use any structure of the matrix $A$.\\

To motivate the problem addressed in this paper, note that such computational problems arise naturally in certain settings. For instance, consider approximating the expected value
\begin{equation*}
\mathbb{E}(f) = \int_{\R^s} f(\bsy^\top) \frac{\exp\left( -\frac{1}{2} \bsy^\top \Sigma^{-1} \bsy \right)}{\sqrt{(2\pi)^s \det(\Sigma)}} \,\mathrm{d} \bsy,
\end{equation*}
where $\Sigma$ is symmetric and positive definite. Using the substitution $\bsy = A^\top \bsx$, where $A$ factorizes $\Sigma$, i.e. $\Sigma = A^\top A$, we arrive at the integral
\begin{equation*}
\mathbb{E}(f) = \int_{\R^s} f(\bsx^\top A) \underbrace{\frac{\exp\left(-\frac{1}{2} \bsx^\top \bsx \right)}{(2\pi)^{s/2} } \rd \bsx}_{=:\rd \mu(\bsx)}. 
\end{equation*}
Such problems arise for instance in statistics when computing expected values with respect to a normal distribution, and in mathematical finance, e.g., for pricing financial products whose payoff depends on a basket of assets. 

Another setting where such problems arise naturally comes from PDEs with random coefficients in the context of uncertainty quantification (see for instance \cite{KN16} for more details). Without stating all the details here, the main computational cost in this context comes from computing
\begin{equation}\label{PDE_example_comp}
D_k = \sum_{j=1}^s x_{j,k} C_j, \quad \mbox{for } k = 0, 1, \ldots, N-1,    
\end{equation}
where $C_j \in \R^{M \times M}$ are matrices (whose size depends on $s$ and $N$). Let $C_j = (c_{j,u,v})_{1\le u,v \le M}$ and define the column vectors $\bsc_{u,v} = (c_{1,u,v}, c_{2,u,v}, \ldots, c_{s,u,v})^\top \in \RR^s$, for $1 \le u, v \le M$. Then we can compute the matrices given by \eqref{PDE_example_comp} by computing
\begin{equation}\label{PDE_example_comp_fast}
X \bsc_{u,v}, \quad \mbox{for } 1 \le u,v \le M.
\end{equation}
In this approach we do not compute the matrices in \eqref{PDE_example_comp} for each $k$ separately, hence this approach requires us to store the results of \eqref{PDE_example_comp_fast} first.

\medskip

It was shown in \cite{DKLS15} that when using particular types of QMC rules, such as (polynomial) lattice rules or Korobov rules,
the cost to evaluate $Q_N(f)$, as given in \eqref{eq:QMC-rule}, can be reduced to only $\calO(\tau \,N \log N)$ operations provided that $\log N \ll s$.
This drastic reduction in computational cost is achieved by a fast matrix-matrix
multiplication exploiting the fact that for the chosen point sets the matrix $X$ can be re-ordered to be of circulant structure.
The fast multiplication is then realized by the use of the fast Fourier transformation (FFT).

Here, we will explore a different method which can also drastically reduce the computation cost of evaluating $Q_N(f)$, as given in \eqref{eq:QMC-rule}.
The reduction in computational complexity is achieved by using point sets which possess a certain repetitiveness in their components.
In particular, the number of different values of the components $x_{k,j}$ (for $0\le k \le N-1$) is in general smaller than $N$
and decreases when $j \in \{1,\ldots,s\}$ increases. As a particular type of such QMC point sets,
we will consider (polynomial) lattice point sets that have been obtained by the so-called reduced CBC construction as in \cite{DKLP15},
and we will also consider similarly reduced versions of digital nets obtained from digital sequences such as Sobol' or Niederreiter sequences.
The corresponding QMC point sets will henceforth be called reduced (polynomial) lattice point sets or reduced digital nets.

\medskip

The idea of our approach, which will be made more precise in the following sections, works as follows.

Assume that we have $N$ samples of the form $(x_{1,k}, x_{2,k}, \ldots, x_{s,k})$, $0 \le k \le N-1$.
We reduce the number of different values by choosing the number of samples differently for each coordinate,
say $N_j$ for the $j$-th coordinate, where $N_j$ divides $N_{j-1}$. E.g., if $N_1 = 4$, $N_2 = 2$, and $N_3 = 1$,
then we generate the points
\begin{equation}\label{ex}
(y_{1,0}, y_{2,0}, y_{3,0}), (y_{1,1}, y_{2,1}, y_{3,0}), (y_{1,2}, y_{2,0}, y_{3,0}), (y_{1,3}, y_{2,1}, y_{3,0}).
\end{equation}
Here, there are $4$ different values for the first coordinate, $2$ different values for the second
coordinate, and the values for the last coordinate are all the same.

What is the advantage of this construction? The advantage can be seen when we compute $\boldsymbol{x}^\top A$.
Let $\bsa_1, \bsa_2, \ldots, \bsa_s$ denote the rows of $A$. If all coordinates are different, we need $\mathcal{O}(Ns)$ operations.
For instance, in the example above we have $4$ points in the $3$-dimensional space, so we need to compute
$$
x_{1,k} \bsa_1 + x_{2,k} \bsa_2 + x_{3,k} \bsa_3, \quad\mbox{for}\quad k \in \{1, 2, 3, 4\}.
$$
However, if we use the points \eqref{ex} then we only need to compute
$$
y_{1,k} \bsa_1 + y_{2, \lfloor k/2 \rfloor} \bsa_2 + y_{3, 0} \bsa_3, \quad\mbox{for}\quad k \in \{1, 2, 3, 4\}.
$$
The last computation can be done recursively, by first computing $y_{3,0} \bsa_3$, then $y_{2,0} \bsa_2 + y_{3,0} \bsa_3$
and $y_{3,1} \bsa_2 + y_{3,0} \bsa_3$, and then finally the remaining vectors. By storing and reusing these intermediate results, we only compute $y_{3,0} \bsa_3$, and $y_{2,0} \bsa_2 + y_{3,0} \bsa_3$ and $y_{2,1} \bsa_2 + y_{3,0} \bsa_3$ once (rather than recomputing the same result as in the straightforward computation).

By applying this idea in the general case, we obtain a similar cost saving as for the fast QMC matrix-vector product in \cite{DKLS15}. However, the present method behaves differently in some situations which can be beneficial. One advantage is that it allows us to use random shifts, which is not possible for the fast QMC matrix-vector product.

\medskip

Before we proceed, we would like to introduce some notation. We will write $\ZZ$ to denote the set of integers, $\ZZ_{\ast}$ to denote the set of integers excluding 0, $\NN$ to denote the positive integers, and $\NN_0$ to denote 
the nonnegative integers. Furthermore, we write $[s]$ to denote the index set $\{1,\ldots,s\}$. 
To denote sets of components we use fraktur font, e.g., $\setu \subseteq [s]$.
 For a vector $\bsx=(x_1,\ldots,x_s)\in [0,1]^s$ and for $\setu \subseteq [s]$, we write $\bsx_\setu=(x_j)_{j \in \setu} \in
[0,1]^{|\setu|}$ and $(\bsx_{\setu},\bszero)\in [0,1]^s$ for the vector $(y_1,\ldots,y_s)$ with $y_j=x_j$ if $j \in \setu$ and $y_j=0$ if $j \not\in
\setu$. For integer vectors $\bsh\in\ZZ^s$, and $\setu\subseteq [s]$, we analogously write $\bsh_{\setu}$ to denote the projection of $\bsh$ onto those components with indices in $\setu$. 

The rest of the paper is structured as follows. Below we introduce lattice rules and polynomial lattice rules and the relevant function spaces. In Section~\ref{sec:reduced_convergence} we state the relevant results on the convergence of the reduced lattice rules. In Section~\ref{sec:fast_reduced_comp} we outline how to use reduced rules for computing matrix products efficiently. In Section~\ref{sec:reduced_nets} we discuss a version of the fast reduced QMC matrix-vector multiplication for digital nets and prove a bound on the weighted discrepancy. In Section~\ref{sec:reduced_MC} we explain how these ideas can also be applied to the plain Monte Carlo algorithm. Numerical experiments in Section~\ref{sec:num_exp} conclude the paper.

\subsection{Lattice point sets and polynomial lattice point sets}

In this section, we would like to give the definitions of the classes of QMC point sets considered in this paper.

\medskip

We start with (rank-1) lattice point sets. 
For further information, we refer to, e.g., \cite{DKP22,DKS13,N92,SJ94} and the references therein.

For a natural number $N \in \NN$ and a vector $\bsz \in \{1, 2, \ldots, N-1\}^s$, a lattice point set consists of
points $\bsx_0,\bsx_1,\ldots,\bsx_{N-1}$ of the form
\begin{equation*}
\bsx_k = \left\{ \frac{k}{N} \bsz \right\}\ \ \ \mbox{ for }\ \ k=0,1,\ldots,N-1.
\end{equation*}
Here, for real numbers $y \ge 0$ we write $\{y\} = y - \lfloor y \rfloor$ for the fractional part of $y$.
For vectors $\bsy$ we apply $\{\cdot \}$ component-wise.

In this paper, we assume that the number of points $N$ is a prime power, i.e., $N=b^m$, with prime $b$ and $m \in \NN$.

\medskip

The second class of point sets considered here are so-called polynomial lattice point sets, whose definition is similar to that of lattice point sets,
but based on arithmetic over finite fields instead of integer arithmetic. %Polynomial lattice point sets are mostly used for QMC methods on Walsh spaces.
To introduce them, let $b$ again be a prime, and denote by $\FF_b$ the finite field with $b$ elements and by $\FF_b[x]$ the set of all polynomials in $x$ with coefficients in $\FF_b$.  We will use a special instance of polynomial lattice point sets over $\FF_b$. For a prime power $N=b^m$ and 
$\bsg=(g_1,\ldots,g_s)\in (\FF_b[x])^s$, 
a polynomial lattice point set consists of $N$ points $\bsx_0,\bsx_1,\ldots,\bsx_{N-1}$ of the form
$$
\bsx_k:=\left(\nu\left(\frac{k(x)\ g_1 (x)}{x^m}\right),\ldots,\nu\left(\frac{ k(x)\ g_s (x)}{x^m}\right)\right)
\ \ \ \mbox{ for }\ k \in \FF_b[x]\ \mbox{ with } \deg(k)<m,
$$
where for $f \in \FF_b[x]$, $f(x)=a_0+a_1 x+\cdots +a_r x^r$, with $\deg(f)=r$, the map $\nu$ is given by
$$
\nu\left(\frac{ f(x)}{x^m}\right):= \frac{a_{\min(r,m-1)}}{b^{m-\min(r,m-1)}}+\cdots+\frac{a_1}{b^{m-1}}+\frac{a_0}{b^m} \in [0,1).
$$
Note that $\nu( f(x)/x^m)=\nu((f(x)\pmod{x^m})/x^m)$. We refer to \cite[Chapter~10]{DP10} for further information on polynomial lattice point sets.

Lattice point sets are used in QMC rules referred to as lattice rules, and analogously for polynomial lattice point sets.

\subsection{Korobov spaces and related Sobolev spaces}\label{sec:spaces}

As pointed out above, lattice point sets are commonly used as node sets in lattice rules, and they are frequently studied in the context of numerical integration of 
functions in Korobov spaces and certain Sobolev spaces, which we would like to describe in the present section. Let us consider first a weighted Korobov space with general weights as studied in
\cite{DSWW06,NW10}. 

In several applications, we may have the situation that different groups of variables have different importance, and this can also
be reflected in the function spaces under consideration.
Indeed, the importance of the different components or groups of components of the functions in the Korobov space
to be defined is specified by a set of positive real numbers $\bsgamma=\{\gamma_{\setu}\}_{\setu \subseteq [s]}$, where we
may assume that $\gamma_{\emptyset}=1$. In this context, larger values of $\gamma_{\setu}$ indicate that the group of variables corresponding to the index set $\setu$ has relatively stronger influence on the computational problem, whereas smaller values of $\gamma_{\setu}$ mean the opposite.

The smoothness of the functions in the space is described by a parameter $\alpha>1/2$. 

Product weights are a common special case of the weights $\bsgamma$ where $\gamma_{\setu} = \prod_{j \in \setu} \gamma_j$ for $u \subseteq [s]$ and where $(\gamma_j)_{j=1, 2, \ldots, s}$ is a sequence of positive real numbers.

The weighted Korobov space, denoted by $\calH(K_{s,\alpha,\bsgamma})$, is a reproducing kernel Hilbert space with kernel function
\begin{eqnarray*}
	K_{s,\alpha,\bsgamma}(\bsx,\bsy)
	&=& 1+\sum_{\emptyset \not=\setu \subseteq [s]} \gamma_{\setu}
	\prod_{j \in \setu}\left(\sum_{h \in \ZZ_{\ast}} \frac{\exp(2 \pi \icomp h(x_j-y_j))}{|h|^{2\alpha}}\right)\\
	&=& 1+ \sum_{\emptyset \not=\setu \subseteq [s]} \gamma_{\setu} \sum_{\bsh_{\setu}\in \ZZ_{\ast}^{|\setu|}}
	\frac{\exp(2 \pi \icomp \bsh_{\setu}\cdot (\bsx_{\setu}-\bsy_{\setu}))}{\prod_{j \in \setu}\abs{h_j}^{2\alpha}}.
\end{eqnarray*}
The corresponding inner product is
\[
\langle f,g\rangle_{K_{s,\alpha,\bsgamma}}=\sum_{\setu \subseteq [s]}
\gamma_{\setu}^{-1} \sum_{\bsh_{\setu}\in \ZZ_{\ast}^{|\setu|}}
\left(\prod_{j \in \setu}\abs{h_j}^{2\alpha}\right) \widehat{f}((\bsh_{\setu},\bszero)) \overline{\widehat{g}((\bsh_{\setu},\bszero))},
\]
where $\widehat{f}(\bsh)=\int_{[0,1]^s} f(\bst) \exp(-2 \pi \icomp \bsh \cdot \bst)\rd \bst$ is the $\bsh$-th Fourier coefficient of $f$. For $\setu = \emptyset$, the empty sum is defined as $\widehat{f}(\bszero) \overline{\widehat{g}(\bszero)}$.

For $h \in \ZZ_{\ast}$, we define $\rho_{\alpha}(h)=|h|^{-2\alpha}$, and for $\bsh=(h_1,\ldots,h_s) \in \ZZ_{\ast}^s$ let
$\rho_{\alpha}(\bsh)=\prod_{j=1}^s \rho_{\alpha}(h_j)$.

It is known (see, e.g., \cite{DSWW06}) that the squared worst-case error of a lattice rule generated by a vector
$\bsz \in \ZZ^s$ in the weighted Korobov space $\calH(K_{s,\alpha,\bsgamma})$ is given by
\begin{equation}\label{eqerrorexprlps}
e_{N,s,\bsgamma}^2 (\bsz)=\sum_{\emptyset\neq\setu\subseteq [s]}\gamma_\setu
\sum_{\bsh_{\setu}\in\calD_\setu}\rho_\alpha (\bsh_\setu),
\end{equation}
where
$$
\calD_\setu = \calD_{\setu}(\bsz) :=\left\{\bsh_\setu\in\ZZ_{\ast}^{\abs{\setu}}\ : \ \bsh_\setu\cdot\bsz_\setu\equiv 0\ (\operatorname{mod}N) \right\}
$$
is called the dual lattice of the lattice generated by $\bsz$.

The worst-case error of lattice rules in a Korobov space can be related to the worst-case error in certain Sobolev spaces.
Indeed, consider a tensor product Sobolev space $\calH_{s,\bsgamma}^{\sob}$
of absolutely continuous functions whose mixed partial derivatives of order $1$
in each variable are square integrable, with norm (see \cite{H98})
$$
 \| f\|_{\calH_{s,\bsgamma}^{\sob}} = \left(\sum_{\setu \subseteq [s]}  \gamma_{\setu}^{-1} \int_{[0,1]^{|\setu|}} \left(\int_{[0,1]^{s-|\setu|}}
 \frac{\partial^{|\setu|}}{\partial \bsx_{\setu}} f(\bsx) \rd \bsx_{[s]\setminus \setu} \right)^2 \rd \bsx_{\setu}\right)^{1/2},
$$
where $\partial^{|\setu|}f/\partial \bsx_{\setu}$ denotes the mixed partial derivative with respect to all variables $j \in \setu$.
As pointed out in \cite[Section~5]{DKS13}, the root mean square worst-case error $\widehat{e}_{N,s,\bsgamma}$ for QMC integration in
$\calH_{s,\bsgamma}^{\sob}$ using randomly shifted lattice rules $(1/N)\sum_{k=0}^{N-1}f\left(\left\{ \frac{k}{N} \bsz+\bsDelta \right\}
\right)$, i.e.,
$$
 \widehat{e}_{N,s,\bsgamma}(\bsz)=\left(\int_{[0,1]^s} e_{N,s,\bsgamma}^2(\bsz,\bsDelta) \rd \bsDelta\right)^{1/2},
$$
where $e_{N,s,\bsgamma}(\bsz,\bsDelta)$ is the worst-case error of
QMC integration in $\calH_{s,\bsgamma}^{\sob}$ using a shifted integration lattice,
is essentially the same as the worst-case error $e_{N,s,\bsgamma}^{(1, \kor)}$ in the weighted Korobov space
$\calH(K_{s,1,\bsgamma})$ using the unshifted version of the lattice rules. In fact, we have
\begin{equation} \label{eq:wceeqwce}
\widehat{e}_{N,s, 2 \pi^2 \bsgamma}(\bsz)=e_{N,s,\bsgamma}^{(1, \kor)}(\bsz),
\end{equation}
where $2 \pi^2 \bsgamma$ denotes the weights $( (2 \pi^2)^{|\setu|} \gamma_{\setu})_{\emptyset \not=\setu \subseteq [s]}$.
For a connection to the so-called anchored Sobolev space see, e.g., \cite[Section~4]{HW00}.

In a slightly different setting, the random shift can be replaced by the tent transformation $\phi(x) = 1 - |1-2x|$ in each variable.
For a vector $\bsx \in [0,1]^s$ let $\phi(\bsx)$ be defined component-wise.
Let $\widetilde{e}_{N,s, \bsgamma}(\bsz)$ be the worst-case error in the unanchored weighted Sobolev space $\calH_{s, \bsgamma}^{\sob}$
using the QMC rule $(1/N)\sum_{k=0}^{N-1}f\left(\phi\left(\left\{ \frac{k}{N} \bsz \right\} \right) \right)$.
Then it is known due to \cite{DNP14} and \cite{CKNS16} that
\begin{equation}\label{eq_wce_tent}
\widetilde{e}_{N,s, \pi^2 \bsgamma}(\bsz) \le e_{N,s,\bsgamma}^{(1, \kor)}(\bsz),
\end{equation}
where $\pi^2 \bsgamma = (\pi^{2|\setu|} \gamma_{\setu})_{\emptyset \neq \setu \subseteq [s]}$, and that the CBC construction
with the quality criterion given by the worst-case error in the Korobov space $\calH(K_{s,1,\bsgamma})$ can be used to construct tent-transformed
lattice rules which achieve the almost optimal convergence order in the space $\calH_{s,\pi^2 \bsgamma}^{\sob}$ under appropriate conditions
on the weights $\bsgamma$ (see \cite[Corollary 1]{CKNS16}). Hence we also have a direct connection between
integration in the Korobov space using lattice rules and integration in the unanchored Sobolev space using tent-transformed lattice rules.

Thus, results shown for the integration error in the Korobov space can, by a few simple modifications, be carried over
to results that hold for anchored and unanchored Sobolev spaces, respectively, by using Equations~(\ref{eq:wceeqwce}) and~\eqref{eq_wce_tent}.

%-----------------------------------------------------------

\subsection{Reduced (polynomial) lattice point sets}\label{sec:reduced_convergence}

In \cite{DKLP15}, the authors introduced so-called reduced lattice point sets and reduced polynomial lattice point sets. The original motivation for
these concepts was to make search algorithms for excellent QMC rules faster for situations where the dependence of a high-dimensional integration
problem on its variable $j$ decreases fast as the index $j$ increases. 
Such a situation might occur in various applications and is modelled by assuming that the weights in the weighted spaces, such as those introduced in
Section \ref{sec:spaces}, decay at a certain speed.

The ``reduction'' in the search for good lattice point sets is achieved by
shrinking the sizes of the sets that the different components of the generating vector $\bsz$ are chosen from. In the present paper, we will
make use of the same idea, but with a different aim, namely that of increasing the speed of computing the matrix product $XA$, as outlined above.

Recall that we assume $N$ to be a prime power, $N=b^m$. A reduced rank-1 lattice point set is obtained by introducing
an integer sequence $\bsw = (w_j)_{j=1}^s \in \N_0^s$ with $0=w_1\le w_2 \le \cdots \le w_s$.
We will refer to the integers $w_j$ as reduction indices. Additionally, for integer $w\ge 0$, we introduce the set
\begin{equation*}
	\mathbb{U}_{b^{m-w}}
	:=
	\begin{cases}
		\{ z \in \{1,2,\ldots,b^{m-w}-1\} : \gcd(z,b)=1 \} & \mbox{if $w<m$,}\\
		\{1\} & \mbox{if $w\ge m$.}
	\end{cases}
\end{equation*}
Note that $\mathbb{U}_{b^{m-w}}$ is the group of units of integers modulo $b^{m-w}$ for $w<m$, and in this case the
cardinality of the set $\mathbb{U}_{b^{m-w}}$ equals $(b-1) b^{m-w-1}$.
For the given sequence $\bsw$ we then define $s^*$ as $s^*:=\max\{j\in\N\colon w_j<m\}$.

The generating vector $\bsz \in \Z^s$ of a reduced lattice rule as in \cite{DKLP15} is then of the form
\begin{equation*}
 \bsz = (b^{w_1} z_1, b^{w_2} z_2,\ldots,b^{w_s} z_s) = (z_1, b^{w_2} z_2,\ldots,b^{w_s} z_s),
\end{equation*}
where $z_j \in \mathbb{U}_{b^{m-w_j}}$ for all $j=1,\ldots,s$. Note that for $j> s^*$ we have $w_j\ge m$ and $z_j=1$.
In this case the corresponding components of $\bsz$ are multiples of $N$. The resulting $b^m$
points of the reduced lattice point set are given by
\begin{eqnarray*}
  \bsx_k &=&
  \left( \left\{ \frac{k z_1 b^{w_1}}{N}\right\}, \left\{ \frac{k z_2 b^{w_2}}{N}\right\}, \ldots, \left\{ \frac{k z_s b^{w_s}}{N} \right\} \right)\\
  &=&\left( \frac{k z_1 \bmod b^{m-\min (w_1,m)}}{b^{m-\min (w_1,m)}},
  \frac{k z_2 \bmod b^{m-\min (w_2,m)}}{b^{m-\min (w_2,m)}}, \ldots, \frac{k z_s \bmod b^{m-\min (w_s,m)}}{b^{m-\min (w_s,m)}} \right)
\end{eqnarray*}
with $k=0,1,\ldots,N-1$. Therefore it is obvious that the components $x_{k,j}$, $k=0,1,\ldots,N-1$ belong to the set
$\{0,1/b^{m-\min(w_j,m)},\ldots,(b^{m-\min(w_j,m)}-1)/b^{m-\min(w_j,m)}\}$ and each of the values is attained exactly
$b^{\min(w_j,m)}$ times for all $j=1,\ldots,s$. In particular, all $x_{k,j}$ equal 0 for $j>s^*$.

Regarding the performance of reduced lattice rules for numerical integration in the Korobov space $\calH(K_{s,\alpha,\bsgamma})$, the following
result was shown in \cite{DKLP15}. For a proof of this result and further background information, we refer to the original paper \cite{DKLP15}.
\begin{theorem}\label{thm:wce_red_lattice}
Let $\bsw=(w_j)_{j=1}^s \in\NN_0^s$ be a sequence of reduction indices, let $\alpha>1/2$, and consider the Korobov space $\calH(K_{s,\alpha,\bsgamma})$.
Using a computer search algorithm, one can construct a generating vector $\bsz=(z_1 b^{w_1},\ldots,z_s b^{w_s})\in\ZZ^s$ such that,
for any $d\in [s]$ and any $\lambda \in(1/(2\alpha),1]$, the following estimate on the squared worst-case error
of integration in $\calH(K_{d,\alpha,\bsgamma})$ holds.
$$
e_{N,s,\bsgamma}^2 ((z_1 b^{w_1},\ldots,z_d b^{w_d}))\le \left(\sum_{\emptyset\neq\setu\subseteq [d]}\gamma_\setu^\lambda
\frac{2(2\zeta (2\alpha\lambda))^{\abs{\setu}}}{b^{\max\{0,m-\max_{j\in \setu} w_j\}}}\right)^{\frac{1}{\lambda}}.
$$
\end{theorem}

Let us briefly illustrate the motivation for introducing the numbers $w_1, w_2, \ldots, w_s$. Assume we have product weights $\gamma_1 \ge  \gamma_2 \ge  \cdots \ge \gamma_s > 0$. We have
\begin{align*}
\sum_{\emptyset\neq\setu\subseteq [d]}\gamma_\setu^\lambda
\frac{2(2\zeta (2\alpha\lambda))^{\abs{\setu}}}{b^{\max\{0,m-\max_{j\in \setu} w_j\}}} \le & b^{-m} \left(  - 1 + 2 \prod_{j = 1}^d \left(1 + \gamma_j^\lambda 2 \zeta(2\alpha \lambda) b^{\min\{m, w_j \} } \right) \right).
\end{align*}
Further assume that we want to have a bound independent of the dimension. In the non-reduced (classical) case we have $w_1 = w_2 = \cdots = w_s = 0$ and hence
\begin{equation*}
\prod_{j=1}^d \left(1 + \gamma_j^\lambda 2 \zeta(2 \alpha \lambda) \right) = \exp\left(\sum_{j=1}^d \log \left(1 + \gamma_j^\lambda 2 \zeta(2\alpha \lambda) \right) \right) \le \exp\left(2 \zeta(2 \alpha \lambda) \sum_{j=1}^d \gamma_j^\lambda \right),
\end{equation*}
where we used $\log (1+z) \le z$ for $z \ge 0$. If $\sum_{j=1}^\infty \gamma_j^\lambda < \infty$, we get a bound which is independent of the dimension $d$. 

For illustration, say $\gamma_j^{1/(2\alpha)} = j^{-4}$, then the infinite sum is finite and we get a bound independent of the dimension. However, a significantly slower converging sequence would still be enough to give us a bound independent of the dimension. So if we introduce $w_1 \le w_2 \le \cdots$, where $w_j = \log_b j^2$ for instance, then we still have
\begin{equation*}
\sum_{j=1}^\infty \gamma_j^\lambda b^{w_j} < \infty.
\end{equation*}
In \cite{DKLP15} we have shown how the $w_j$ can be used to reduce the construction cost of the CBC construction by reducing the  
size of the search space from $b^m$ to $b^{\max\{0, m-w_j\}}$ in component $j$. In this paper we show that the $w_j$ can also be used to reduce the computation cost of computing $X A$, where the rows of $X$ are the lattice points of a reduced lattice rule. The speed-up which can be achieved this way will depend on the weights $\{\gamma_{\setu}\}_{\setu \subseteq [s]}$ (and the $w_1 \ge w_2 \ge \cdots$). This is different from the fast QMC matrix-vector product in \cite{DKLS15}, which works independently of the weights and does not influence tractability properties.

It is natural to expect that one can use an analogous approach for polynomial lattice rules leading to similar results.

\section{The fast reduced matrix product computation}\label{sec:fast_reduced_comp}

\subsection{The basic algorithm}\label{sec:basic_alg}

We first present some observations which lead us to an efficient algorithm for computing $XA$.

Let $X = [\bsx_0^\top, \bsx_1^\top, \ldots, \bsx_{N-1}^\top]^\top$ be the $N\times s$-matrix whose $k$-th row is the $k$-th point of the reduced lattice point set (written as a row vector). Let $\boldsymbol{\xi}_j$ denote the $j$-th column of $X$, i.e. 
$X = [\boldsymbol{\xi}_1, \boldsymbol{\xi}_2, \ldots, \boldsymbol{\xi}_s]$. Let $A = [\bsa_1, \bsa_2, \ldots, \bsa_s]^\top$, 
where $\bsa_j \in \R^{1 \times \tau}$ is the $j$-th row of $A$. Then we have
\begin{equation}\label{Eq_XA_product}
X A = [\boldsymbol{\xi}_1, \boldsymbol{\xi}_2, \ldots, \boldsymbol{\xi}_s]  \begin{pmatrix} \bsa_1 \\ \bsa_2 \\ \vdots \\ \bsa_s \end{pmatrix} 
= \boldsymbol{\xi}_1 \bsa_1 + \boldsymbol{\xi}_2 \bsa_2 + \cdots + \boldsymbol{\xi}_{s} \bsa_{s}.
\end{equation}

In order to illustrate the inherent repetitiveness of a reduced lattice point set, consider a reduction index $0 < w_j < m$
and the corresponding component $z_j$ of the generating vector.  The $j$-th component of the $N=b^m$ points of the reduced lattice point set 
(i.e., the $j$-th column $\boldsymbol{\xi}_j$ of $X$) is then given by
\begin{eqnarray*}
\boldsymbol{\xi}_j & := &	\left( \frac{0\cdot z_j \bmod b^{m-w_j}}{b^{m-w_j}}, \frac{1 \cdot z_j \bmod b^{m-w_j}}{b^{m-w_j}}, 
	\ldots, \frac{(b^m-1)\cdot z_j \bmod b^{m-w_j}}{b^{m-w_j}} \right)^\top
	\\
	&=&  \underbrace{\left(X_j,\ldots,X_j \right)^\top}_{b^{w_j} \text{ times}}, 
\end{eqnarray*}
where 
$$
X_j = \left( 0, \frac{z_j \bmod b^{m-w_j}}{b^{m-w_j}}, \ldots, \frac{(b^{m-w_j}-1) z_j \bmod b^{m-w_j}}{b^{m-w_j}} \right)^{\top}.
$$
We will exploit this repetitive structure within the reduced lattice points to derive a fast matrix-vector multiplication algorithm.

\medskip

Based on the above observations, it is possible to formulate the following algorithm to compute \eqref{Eq_XA_product} in an efficient way.
Note that for $j>s^*$ the $j$-th column of
$X$ consists only of zeros, so there is nothing to compute for the entries of $X$ corresponding to these columns.

\begin{algorithm}[H]
\caption{Fast reduced matrix product}
\label{alg:fast-mv-prod}
\vspace{5pt}
\textbf{Input:} Matrix $A \in \R^{s \times \tau}$, integer $m \in \N$, prime $b$, reduction indices $0=w_1\le w_2 \le \cdots \le w_s$, 
corresponding generating vector of reduced lattice rule, $\bsz=(z_1, b^{w_2} z_2,\ldots,b^{w_s} z_s)$. \\
\vspace{-10pt}
\begin{algorithmic}
	\STATE Set $N=b^m$ and set $P_{s^*+1} = \bszero_{1 \times \tau} \in \R^{1 \times \tau}$.
	\FOR{$j=s^*$ {\bf to} $1$}
	\STATE $\bullet$ Compute the $b^{m-w_j}$ reduced lattice points
	\begin{equation*}
	X_j = \left( 0, \frac{z_j \bmod b^{m-w_j}}{b^{m-w_j}}, \ldots, \frac{(b^{m-w_j}-1) z_j
	\bmod b^{m-w_j}}{b^{m-w_j}} \right)^{\top} \in \R^{ b^{m-w_j} \times 1} .
	\end{equation*}
	\vspace{-5pt}
	\STATE $\bullet$ Compute $P_j$ as
	\begin{equation*}
	P_j =
	\rotatebox[origin=c]{90}{$b^{\min(w_{j+1},m)-w_j}$ times}\left\{
	\begin{pmatrix}
	P_{j+1} \\
	P_{j+1} \\
	\vdots \\
	P_{j+1} \\
	\end{pmatrix}
	\right.
	+
	X_j  \bsa_j \in \R^{b^{m-w_j} \times \tau},
	\end{equation*}
	\hspace{5pt} where $\bsa_j \in \R^{1 \times \tau}$ denotes the $j$-th row of the matrix $A$.
	\ENDFOR
	\STATE Set $P = P_1$.
\end{algorithmic}
\vspace{5pt}
\textbf{Return:} Matrix product $P = X A$.
\end{algorithm}

The following theorem gives an estimate of the computational cost of Algorithm \ref{alg:fast-mv-prod}, which shows that
by using a reduced point set we can obtain an improved computation time over that in \cite{DKLS15}, which only depends on
the index $s^*$, but not on $s$ anymore.
\begin{theorem}
Let a matrix $A \in \R^{s \times \tau}$, an integer $m \in \N$, a prime $b$,
and reduction indices $0=w_1 \le w_2 \le \cdots \le w_s$ be given. Furthermore,
let $\bsz=(z_1, b^{w_2} z_2,\ldots,b^{w_s} z_s)$ be the generating vector of a 
reduced lattice rule corresponding to $N=b^m$ and the given reduction indices $(w_j)_{j=1}^s$.
Then the matrix product $P=X A$ can be computed via Algorithm \ref{alg:fast-mv-prod} using
\begin{equation*}
\calO\left( \tau \, N \sum_{j=1}^{s^*} b^{-w_j} \right)
\end{equation*}
operations and requiring $\calO(N \tau )$ storage.
Here, $X$ is the $N \times s$-matrix whose rows are the $N$ reduced lattice points.
\end{theorem}
\begin{proof}
In the $j$-th step the generation of the $b^{m-w_j}$ lattice points requires $\calO(b^{m-w_j})$ operations and storage. The most costly
operation in each step is the product
$X_j  \bsa_j$ which requires $\calO(b^{m-w_j} \, \tau  )$ operations, but this step only needs to be carried out for those $j$ with
$j\le s^*$. Summing over all $j=1,\ldots,s^*$, the computational complexity amounts to
\begin{equation*}
\calO\left( \sum_{j=1}^{s^*} b^{m-w_j} \, \tau \right) = \calO\left(\tau  \, b^m \sum_{j=1}^{s^*} b^{-w_j} \right)
\end{equation*}
operations. Furthermore, storing the matrix $P_j$ requires $\calO(b^{m-w_j} \, \tau )$ space, which attains a maximum of $\calO(b^{m} \, \tau )$
for $w_1=0$. Note that in an efficient implementation the matrices $P_j$ are overwritten in each step and do not all have to be stored.
\end{proof}

In the next section we discuss the fast reduced QMC matrix-vector product where the number of points is a power of 2.

\subsection{An optimized algorithm}

Recall that, for a sequence $\bsw$ of reduction indices, we have $s^*=\max\{j\in\N\colon w_j<m\}$. 
Since for $j>s^*$ the $j$-th column of $X$ consists only of zeros, we can restrict our considerations in this section to
the product $\widetilde{X} \widetilde{A}$, where $\widetilde{X}$ is an $N \times s^*$-matrix, and $\widetilde{A}$ is an $s^* \times \tau$-matrix.

Assume that $0 = w_1 \le w_2 \le \cdots \le w_{s^*} <  m$ and define, for $I \in \{0,1,\ldots,m-1\}$, the quantity
$$
\tau_I :=\# \{ j \in \{1,\ldots,s^*\} \mid w_j=I \},
$$
which denotes the number of $w_j$ which equal $I$. Obviously, we then have that $\sum_{I=0}^{m-1} \tau_I = s^*$.

Consider then the following alternative fast reduced matrix product algorithm.
\begin{algorithm}[H]
	\caption{Optimized fast reduced matrix product}
	\label{alg:fast-mv-prod-alt}
	\vspace{5pt}
	\textbf{Input:} Matrix $\widetilde{A} \in \R^{s^* \times \tau}$, integer $m \in \N$, prime $b$, reduction indices
	$0=w_1\le w_2 \le \cdots \le w_{s^*} < m$, 
	the corresponding generating vector of a reduced lattice rule, $\bsz=(z_1, b^{w_2} z_2,\ldots,b^{w_{s^*}} z_{s^*})$. \\
	\vspace{-10pt}
	\begin{algorithmic}
		\STATE Set $N=b^m$, set $\widetilde{P}_{m} = \bszero_{1 \times \tau} \in \R^{1 \times \tau}$, and $w_{s^*+1}=m$.
		\FOR{$I=m-1$ {\bf to} $0$}
		\STATE $\bullet$ Compute the matrix
		\begin{equation*}
		\widetilde{X}_I = (W_1^I,\ldots,W^I_{\tau_I}) \in \R^{b^{m-I} \times \tau_I},
		\end{equation*}
		whose columns are the reduced lattice points
		\begin{equation*}
		W_{r}^I = \left( 0, \frac{z_{j_r} \bmod b^{m-I}}{b^{m-I}}, \ldots,
		\frac{(b^{m-I}-1) z_{j_r} \bmod b^{m-I}}{b^{m-I}} \right)^\top \in \R^{b^{m-I} \times 1},\ 1\le r\le \tau_I,
		\end{equation*}
		and where the $j_r$, $1\le r\le \tau_I$, are those indices for which $w_{j_r}=I$. 
		If $\tau_I = 0$, then set $\widetilde{X}_I = \bszero_{b^{m-I}\times \tau_I}$.
		\STATE $\bullet$ Compute $\widetilde{P}_{I}$ as \begin{equation*}
		\widetilde{P_{I}}=\rotatebox[origin=c]{90}{$b$ times}\left\{
		\begin{pmatrix}
		\widetilde{P}_{I+1} \\
		\vdots \\
		\widetilde{P}_{I+1} \\
		\end{pmatrix}
		\right.
		+  \widetilde{X}_I  \widetilde{A}_I \in \R^{b^{m-I} \times \tau},
		\end{equation*}
		where $\widetilde{A}_I \in \R^{\tau_I \times \tau}$ denotes the rows of the matrix $\widetilde{A}$ that correspond to the $j$ with $w_j=I$. 
		If $\tau_I = 0$, then set $\widetilde{A}_I = \bszero_{\tau_I \times \tau} \in \R^{\tau_I \times \tau}$.   
		\ENDFOR
		\STATE Set $\widetilde{P} = \widetilde{P}_0 $.
	\end{algorithmic}
	\vspace{5pt}
	\textbf{Return:} Matrix product $\widetilde{P} = \widetilde{X} \widetilde{A}$.
\end{algorithm}
The next theorem provides an estimate on the computation time of Algorithm \ref{alg:fast-mv-prod-alt}, which again is independent of $s$.
\begin{theorem}
Let a matrix $A \in \R^{s \times \tau}$, an integer $m \in \N$, a prime $b$, and reduction indices $0=w_1\le w_2 \le \cdots \le w_{s^*} < m$ be given.
Furthermore, let $\bsz=(z_1, b^{w_2} z_2,\ldots,b^{w_s} z_s)$ be the generating vector of a reduced lattice rule corresponding to $N=b^m$
and the given reduction indices $(w_j)_{j=1}^s$. Then the matrix product $P=X A$ can be computed via Algorithm \ref{alg:fast-mv-prod-alt} using
\begin{equation*}
\calO\left(\tau \, N m \right)
\end{equation*}
operations and requiring $\calO(N \tau )$ storage.
Here, $X$ is the $N \times s$-matrix whose rows are the $N$ reduced lattice points.
\end{theorem}
\begin{proof}
As outlined above, it is no relevant restriction to reduce the matrices $X$ and $A$ to an $N\times s^*$-matrix $\widetilde{X}$ and an
$s^* \times \tau$-matrix $\widetilde{A}$, respectively, and then apply Algorithm \ref{alg:fast-mv-prod-alt}.

In the $I$-th step of the algorithm, the generation of the $\tau_I b^{m-I}$ lattice points requires $\calO(\tau_I b^{m-I})$
operations and storage. The most costly operation in each step is the product $\widetilde{X}_I \widetilde{A}_I$
which, via the fast QMC matrix product in \cite{DKLS15}, requires
$\calO((m-I) \, b^{m-I} \, \tau )$ operations. Summing over all $I=0,\ldots,m-1$, the computational complexity amounts to
\begin{equation*}
\calO\left( \sum_{I=0}^{m-1} (m-I) \, b^{m-I} \, \tau \right) = \calO\left(\tau \, b^m \sum_{I=0}^{m-1} \frac{m-I}{b^{I}} \right)
=\calO\left(\tau \, b^m \frac{b^2 m}{(b-1)^2} \right)
=\calO\left(\tau \, N \, m \right)
\end{equation*}
operations. Furthermore, storing the matrix $\widetilde{P}_j$ requires $\calO(b^{m-I} \, \tau )$ space,
which attains a maximum of $\calO(b^{m} \, \tau )$ for $I=0$. Note that in an efficient implementation
the matrices $\widetilde{P}_j$ are overwritten in each step and do not all have to be stored.
\end{proof}

\subsection{Transformations, shifting, and computation for transformation functions}

In applications from mathematical finance or uncertainty quantification, the integral to be approximated is often not over the unit cube but over $\R^s$ with respect to a normal distribution. In order to be able to use lattice rules in this context, one has to apply a transformation and use randomly shifted lattice rules. In the following we show that the fast reduced QMC matrix-vector product can still be used in this context.

We have noted before that projections $P_j = \pi_j(P)$ of a reduced lattice point set 
 $P$ onto the $j$-th component possess a repetitive structure, that is,
$$
	P_j
	=\underbrace{\left(X_j,\ldots,X_j \right)^{\top}}_{b^{\min(w_j,m)} \text{ times}}
$$
with
$$
	X_j = \left( 0, \frac{z_j \bmod b^{m- \min(w_j,m)}}{b^{m- \min(w_j,m)}}, \ldots, 
	\frac{(b^{m- \min(w_j,m)}-1) z_j \bmod b^{m- \min(w_j,m)}}{b^{m- \min(w_j,m)}} \right)^{\top}.
$$
This repetitive structure is preserved when applying a mapping $\varphi:[0,1] \to \R$ elementwise to the projection $P_j$ since
$$
	\varphi(P_j)=
	\underbrace{\left(\varphi(X_j),\ldots,\varphi(X_j) \right)^{\top}}_{b^{ \min(w_j,m)} \text{ times}}.
$$
This approach also works for the map $\psi:[0,1] \to [0,1]$ with $\psi(x)=\{x + \Delta\}$, i.e, for shifting of the lattice points modulo one.
In particular, this observation holds for componentwise maps of the form $\varphi:[0,1]^s \to \R$ with
$\varphi(\bsx)=(\varphi_1(x_1),\ldots,\varphi_s(x_s))$ that are applied simultaneously to all $N$ elements of the lattice point set.
For a map of this form Algorithm~\ref{alg:fast-mv-prod} can be easily adapted by replacing $X_j$ by $\varphi_j(X_j)$. If we
wish to apply Algorithm~\ref{alg:fast-mv-prod-alt} instead, the matrices $\widetilde{X}_I$ can be replaced by the correspondingly
transformed matrices, however, the fast reduced QMC matrix vector product can only be used here if all components with indices in $\tau_I$ use the same transformation.

\section{Reduced digital nets}\label{sec:reduced_nets}

In this section we present a reduced point construction for so-called digital $ (t,m,s) $-nets. Typical examples are digital nets derived from Sobol', Faure, and Niederreiter sequences. In general, a $(t,m,s)$-net is defined as follows.
 
Given an integer $b\ge 2$, an elementary interval in $[0,1)^s$ is an interval of the form
$
\prod_{j=1}^s [a_jb^{-d_j},(a_j+1)b^{-d_j})
$
where $a_j,d_j$ are nonnegative integers with $0 \le a_j < b^{d_j}$ for $1 \le j \le s$.
Let $t,m$, with $0\le t\le m$, be integers. Then a $(t,m,s)$-net in
base $b$ is a point set $P_m$ in $[0,1)^s$ with $b^m$ points such that any elementary
interval in base $b$ with volume $b^{t-m}$ contains exactly $b^t$
points of $P_m$.

Note that a low $t$-value of a $(t,m,s)$-net implies better equidistribution properties and usually also better error bounds for integration rules based on such nets.
How to find nets with low $t$-values is an involved question, see, e.g., \cite{N92,DP10}. Due to the important role of the $t$-value, one sometimes also considers a 
slightly refined notion of a $(t,m,s)$-net, which is then referred to as a $((t_\fraku)_{\fraku \subseteq [s]},m,s)$-net. The latter notion means that for any $\setu\neq\emptyset$, $\setu\subseteq [s]$, the projection of the net is a $(t_\fraku,m,\abs{\setu})$-net.

The most common method to obtain $(t,m,s)$-nets are so-called digital constructions, yielding digital $(t,m,s)$-nets. These work as follows. Let $b$ be a prime number and recall that $\FF_b$ denotes the finite field with $b$ elements. We identify this set with the integers $\{0, 1, \ldots, b-1\}$. We denote the (unique) $b$-adic digits of some $n \in \bbN$ by $\Vec{n} \in \bbF_b^\bbN$, ordered from the least significant, that is $n = \Vec{n}  \cdot (1,b,b^2, \ldots)$. Here, the sum is always finite as there are only finitely many non-zero digits in $\Vec{n}$. Thus, with a slight abuse of notation we write $ \Vec{n} \in \bbF_b^m $ if $n < b^m$.
Analogously, we denote the $b$-adic digits of $y \in [0,1)$ by $\Vec{y} \in \bbF_b^\bbN$, i.e. $y = \Vec{y}  \cdot (b^{-1},b^{-2}, \ldots)$, with the additional constraint that $\Vec{y}$ does not contain infinitely many consecutive entries equal to $b-1$.
    
Given \emph{generating matrices} $ C^{(j)} =\left(C^{(j)}_{p,q}\right)_{p,q=1}^m \in \mathbb{F}_b^{m\times m} $ for $ j = 1,\ldots,s $, a 
\emph{digital net} is defined as $ P_m(\{C^{(j)}\}_j) := \{\bsy_0,\ldots,\bsy_{b^m-1}\} $, where $ \bsy_n = (y_{1,n},\ldots,y_{s,n}) \in [0,1)^s, $ 
\begin{equation}\label{def:digital}
 \Vec{y}_{j,n} :=  C^{(j)}\Vec{n} \quad \mbox{and } y_{j,n} = \Vec{y}_{j,n} \cdot (b^{-1}, b^{-2}, \ldots, b^{-m}).
\end{equation}
From this definition, given reduction indices $ \bsw $, one can construct a reduced digital net by setting the last 
$ \min(w_j,m)$ rows of $C^{(j)}$ to $\bszero$. To be more precise,
\begin{equation}\label{reduced_genmat_net2}
\widehat{C}^{(j)}_{p,q} :=
\begin{cases}
C^{(j)}_{p,q} & \text{ if } p \in \{ 1,\ldots, m- \min(w_j,m)  \},   \\
0 & \text{ if } p \in \{ m - \min(w_j,m) + 1, \ldots, m \},
\end{cases}
\end{equation}
and applying \eqref{def:digital} with the latter choice $ \widehat{C}^{(j)} =\left(\widehat{C}^{(j)}_{p,q}\right)_{p,q=1}^m $ of the generating matrices. 
Note that $ \widehat{C}^{(j)} $ is just the zero matrix if $j>s^*$. For the reduced digital net, we then write 
$P_m(\{\widehat{C}^{(j)}\}_j) := \{\bsz_0,\ldots,\bsz_{b^m-1}\}$.

The construction \eqref{reduced_genmat_net2} allows us to generate a reduced digital net for any given digital net.

\begin{algorithm}[H]
\caption{Computation of a reduced digital net}
\label{alg:red_tsnet}
\vspace{5pt}
\textbf{Input:} Prime $b$, generating matrices $C^{(j)} \in \bbF_b^{m \times m}$, $j = 1,\ldots,s$,  reduction indices $0=w_1\le w_2 \le \cdots \le w_s$.\\
\vspace{-10pt}
\begin{algorithmic}
\STATE Set $\bsy_0 = (0,\ldots,0)$
\FOR{$j=1,\ldots,s$}
\FOR{$n = 0, 1,\ldots, b^m-1$}
\STATE  Compute $\Vec{z}_{j,n} = \widehat{C}^{(j)} \Vec{n}$, with the choice $ \widehat{C}^{(j)} $ from \eqref{reduced_genmat_net2}.
\STATE  Set $z_{j,n}= \Vec{z}_{j,n}  \cdot (b^{-1},b^{-2}, \ldots)$
\ENDFOR
\ENDFOR
\end{algorithmic}
\vspace{5pt}
\textbf{Return:} $ P_m(\{\widehat{C}^{(j)}\}_j) = \{\bsz_n = (z_{1,n} \ldots, z_{s,n})\in [0,1)^s  \colon n = 0, \ldots,b^m-1\}$.
\end{algorithm}
    
The advantage of using a reduced digital net is that in component $j$ all the values of the $z_{j,n}$ are in the set 
$\{0, 1/b^{m- \min(w_j,m)}, 2/b^{m- \min(w_j,m)}, \ldots, 1-1/b^{m- \min(w_j,m)}\}$. Since the digital net has $b^m$ points, the values necessarily repeat $b^{\min(w_j,m)}$ times. This can be used to achieve a reduction in the computation of $X A$ in the following way.

\begin{algorithm}[H]
\caption{Fast reduced matrix product for digital nets}
\label{alg:fast-mv-prod-dignet}
\vspace{5pt}
\textbf{Input:} Matrix $A \in \R^{s \times \tau}$ with $j$-th row vector $\bsa_j$, $j = 1, 2, \ldots, s$, integer $m \in \N$, prime $b$, reduction indices $0=w_1\le w_2 \le \cdots \le w_s$. Let $s^\ast \le s$ be the largest index such that $w_{s^\ast} < m$. Let $\{\bsy_n = (y_{n,1}, y_{n,2}, \ldots, y_{n,s})^\top \in [0,1)^s: n = 0, 1, \ldots, b^m-1\}$ be a digital net.

%\vspace{-10pt}
\begin{algorithmic}
    \STATE Set $P_{s^\ast+1} = \bszero \in \R^{b^m \times \tau}$. 
	\FOR{$j=s^*$ {\bf to} $1$}
    \STATE $\bullet$ Compute the row vectors
    \begin{equation*}
       \bsc_k = \frac{k}{b^{m-w_j}} \bsa_j, \quad k = 0, 1, \ldots, b^{m-w_j}-1.
    \end{equation*}
	\vspace{-5pt}
	\STATE $\bullet$ Compute $P_j$ as
	\begin{equation*}
	P_j = P_{j+1} + 
	\begin{pmatrix}
	    \bsc_{\lfloor y_{0,j}b^{m-w_j} \rfloor} \\
	    \bsc_{\lfloor y_{1,j}b^{m-w_j} \rfloor} \\
	    \vdots \\
	    \bsc_{\lfloor y_{b^m-1,j} b^{m-w_j} \rfloor}
	\end{pmatrix}\in \R^{b^{m} \times \tau}.
	\end{equation*}
	\ENDFOR
	\STATE Set $P = P_1$.
\end{algorithmic}
\vspace{5pt}
\textbf{Return:} Matrix product $P = X A$.
\end{algorithm}

Compared with computing $XA$ directly, Algorithm~\ref{alg:fast-mv-prod-dignet} reduces the number of multiplications from $\mathcal{O}(\tau b^m)$ to $\mathcal{O}(\tau b^{m-w_j})$ in coordinate $j$ and to $\mathcal{O}(\tau b^m \sum_{j=1}^{s^\ast} b^{-w_j})$ overall compared to $\mathcal{O}(s \tau b^m)$. The number of additions is the same in both instances. 

The difference here to the approach for lattice point sets is that although component $j$ has $b^{w_j}$ repeated values, the repeating pattern in each component is different and so when we add up the vectors resulting from the different components, we do not have repetitions in general and so we do not get a reduced number of additions. The analogue to the method in Section \ref{sec:basic_alg} for lattice point sets applied to digital nets would be to delete columns of $C^{(j)}$ (rather than rows as we did in this section). The problem with this approach is that if we delete columns, then the $(t,m,s)$-net property of the digital net is not guaranteed anymore. A special construction of digital $(t,m,s)$-nets with additional properties would be needed in this case.

For the case of reduced lattice point sets, we can use Theorem \ref{thm:wce_red_lattice} to obtain an error bound on the performance of the 
corresponding QMC rule when using \eqref{eq:QMC-rule} to approximate \eqref{eq:integral}. For the case of reduced digital nets, there is no 
existing error bound analogous to Theorem \ref{thm:wce_red_lattice}. We outline the error analysis in the subsequent section.

\subsection{Error analysis}

Consider  the case of digital nets from Algorithm \ref{alg:red_tsnet}. For this we fix $ m\in\NN $.
The \emph{weighted star discrepancy} is a measure of the worst-case quadrature error for a node set $P_m$, with $b^m$ nodes, defined as
\begin{equation}
    D_{b^m,\bsgamma}^{*}(P_m) := \sup_{x \in (0,1]^s} \max_{\emptyset \neq \fraku \subseteq [s]} \gamma_{\fraku} \abs{\Delta_{P_m,\fraku}(\bsx)}, 
\end{equation}
where
\begin{equation}\label{def:DeltaP}
    \Delta_{P_m,\fraku}(\bsx) := \frac{ \# \{(y_1,\ldots,y_s)\in P_m  \colon y_j < x_j, \, \forall j \in \fraku \}}{b^m}- \prod_{j\in \fraku} x_j.
\end{equation}
We additionally write $\Delta_{P_m}(\bsx) = \Delta_{P_m,[s]}(\bsx)$.
For all $k \in \{0,\ldots,b^m-1\} $, define $\Vec{k} = (\Vec{ \kappa}_0,\ldots,\Vec{\kappa}_{m-1})\in \bbF_b^m$ 
its vector of $b$-adic digits, ordered from the least significant to the most significant. Moreover, define
\begin{equation*}
    \rho(k) = \begin{cases} 1, & \mbox{if } k = 0, \\ \frac{1}{b^{r} \sin(\pi \kappa_{r-1} / b)}, 
    & \mbox{if } k = \kappa_0 + \kappa_1 b + \cdots + \kappa_{r-1} b^{r-1}, \\ 
    & \mbox{with } \kappa_0, \ldots, \kappa_{r-2} \in \{0, 1, \ldots, b-1\}, \kappa_{r-1} \in \{1, \ldots, b-1\}. \end{cases}
\end{equation*}

In the following proposition, we prove a bound on $\Delta_{P_m,\fraku}(\bsx)$.

\begin{proposition}\label{prop:reduced_discr}
Let $\widehat{P}_m:= P_m(\{\widehat{C}^{(j)}\}_j) = \{\bsz_0,\ldots,\bsz_{b^m-1}\} $ be generated by Algorithm \ref{alg:red_tsnet}, and let $\bsx\in (0,1]^s$.  Let $0 = w_1 \le w_2 \le \cdots \le w_s$ and let $s^\ast \in [s]$ be the largest index such that $w_{s^\ast} < m$. Then for any $\fraku \subseteq [s]$ with $\fraku \neq \emptyset$ we have%, $\fraku\subseteq [s]$, and let $\frakv = \fraku \cap [s^\ast] $.
\begin{equation*}
\left| \Delta_{\widehat{P}_m,\fraku}(\bsx) \right|  \le  
\begin{cases}1 & \mbox{if } \fraku \not \subseteq [s^\ast], \\ 
{\displaystyle 1 - \prod_{j\in\fraku} \left ( 1 - \frac{1}{b^{m-w_j}} \right )  + \sum_{ \substack{\bsk_{\fraku} \in \NN_0^{\abs{\fraku}} \setminus \{\bm{0}\} \\ k_j \in \{0,\ldots, b^{m-w_j}-1\} \\ \sum_{j\in\fraku} (\widehat{C}^{(j)})^\top \Vec{k}_j \equiv \Vec{0} \pmod{b}} } \prod_{j\in\fraku} \rho(k_j),} & \mbox{if } \fraku \subseteq [s^\ast], \\  \end{cases} 
\end{equation*}
where $\widehat{C}^{(j)}$ is defined in \eqref{reduced_genmat_net2}. 
\end{proposition}
To prove Proposition \ref{prop:reduced_discr}, we need the next elementary lemma, extending \cite[Lemma 3.18]{DP10}, 
which can be verified by induction on $s$.
\begin{lemma}\label{lem:diffprod}
Let $J$ be a finite index set and assume $u_j,v_j \in [0,1] $, $ |u_j-v_j| \le \delta_j \in [0,1] $ for all $ j \in J$. Then
\begin{equation*}
    \abs{\prod_{j\in J} u_j - \prod_{j\in J} v_j} \le 1 - \prod_{j\in J} (1 - \delta_j) \le \sum_{j\in J} \delta_j .
\end{equation*}
\end{lemma}
\begin{proof}[Proof of Proposition \ref{prop:reduced_discr}]

The bound for the case when $\fraku \not\subseteq [s^\ast]$ is trivial. Hence we can now focus on the case when $\fraku \subseteq [s^\ast]$. We operate along the lines of the proof of \cite[Theorem 3.28]{DP10}. We define the mapping $T:[0,1]^{|\fraku|} \to [0,1]^{|\fraku|}$ given by 
$$
T(\bsx_\fraku) = T((x_j)_{j \in \fraku}) = ((T_{m-w_j}(x_j))_{j \in \fraku}),
$$ 
where $T_v(x) = \lceil x b^{v} \rceil b^{-v}$. 

Now, let us assume that $\bsx_\fraku = (x_j)_{j \in \fraku} \in (0,1]^{|\fraku|}$ has been chosen arbitrarily but fixed. 
For short, we write $\overline{\bsx}_\fraku =(\overline{x}_j)_{j \in \fraku} := T(\bsx_\fraku)$. %Note that, as all $x_j>0$, that for $j>s^*$ we have  $\overline{x}_j=T_{m-\min(w_j,m)}(x_j)=T_0 (x_j)= \lceil x_j \rceil =1$.

Recall that, by the definition of the matrices $\widehat{C}^{(j)}$, $j\in\{1,\ldots,s\}$, 
the points $\bsz_n=(z_{1,n},\ldots,z_{s,n})$ are such that the $z_{j,n}$ have at most $m-\min (w_j,m)$ non-zero digits.

Using the triangle inequality we get
\begin{equation*}
    \abs{\Delta_{\widehat{P}_m,\fraku}(\bsx)} \le  \abs{\Delta_{\widehat{P}_m,\fraku}(\bsx)- \Delta_{\widehat{P}_m,\fraku}(\overline{\bsx})} 
    + \abs{\Delta_{\widehat{P}_m,\fraku}(\overline{\bsx})}.
\end{equation*}
For any $\bsz_n \in \widehat{P}_m$, we denote the $b$-adic digits of the $j$-th component $z_{j,n}$ by $z_{j,n,i}$, $i\in\{1,\ldots,m\}$.
By construction, we see that $z_{j,n,i} = 0$ for $ i > m- \min\{w_j,m\} $. 
Hence 
\begin{equation*}
\widehat{P}_m \subseteq \left\{\left(h_1 b^{-(m- \min\{w_1,m\})},\ldots,h_s b^{-(m- \min\{w_s,m\})}\right) \colon h_j \in \NN_0 \right\}. 
\end{equation*}
This implies, as $\overline{\bsx}_\fraku  =T(\bsx_\fraku )$,  
\begin{equation*}
\# \{(z_1,\ldots,z_s)\in \widehat{P}_m \colon z_j < x_j, \, \forall j\in\fraku  \} 
= \# \{(z_1,\ldots,z_s)\in \widehat{P}_m \colon z_j < \overline{x}_j, \, \forall j\in\fraku \},
\end{equation*}
and thus
\begin{equation*}
   \abs{\Delta_{\widehat{P}_m, \fraku}(\bsx) - \Delta_{\widehat{P}_m, \fraku}(\overline{\bsx})}  = \abs{\prod_{j\in\fraku} x_j - \prod_{j\in\fraku} \overline{x}_j}
   =\prod_{j\in\fraku} \overline{x}_j - \prod_{j\in\fraku} x_j  \le 1 - \prod_{j \in \fraku} \left(1 - \frac{1}{b^{m-w_j}} \right),
\end{equation*} 
where we used Lemma~\ref{lem:diffprod} for the last inequality.

Since $ [\bm{0},\overline{\bsx}) $ is a disjoint union of intervals of the form 
$$ 
J = \prod_{j=1}^{s} [\frac{h_j}{b^{m-\min(w_j,m)}},\frac{h_j+1}{b^{m-\min(w_j,m)}}), 
$$ 
an application of \cite[Lemma 3.9]{DP10} implies that
$ \widehat{\chi_{[\bm{0},\overline{\bsx})}}(\bsk) = 0 $ for all $ \bsk \in \bbN_0^s\setminus\{\bm{0}\} $ such that $ k_j \ge b^{m-\min(w_j,m)} $ for at least one $j$. Here $ \chi_{[\bm{0},\overline{\bsx})} $ denotes the indicator function and $\widehat{\chi_{[\bm{0},\overline{\bsx})}}(\bsk)$ 
are the corresponding Walsh coefficients (we use similar notation as in \cite[Chapter~2]{DP10}). The complete analogue of this 
observation holds if we consider the projection of $J$, given by 
$J_{\setu} = \prod_{j\in\setu} [\frac{h_j}{b^{m-\min(w_j,m)}},\frac{h_j+1}{b^{m-\min(w_j,m)}})$, the projections 
$\overline{\bsx}_{\setu}$ and $\bsk_{\setu}$ of $\overline{\bsx}$ and $\bsk$, respectively, and the projections $\bsz_{n,\setu}$ of the points $\bsz_n$ in $\widehat{P}_m$. 
Then, \cite[Lemmas 3.29 and 4.75]{DP10} yield
\begin{align*}
    \abs{\Delta_{\widehat{P}_m,\fraku}(\overline{\bsx})} &= \abs{\frac{1}{b^m}\sum_{\substack{\bsk_{\fraku} \in \NN_0^{\abs{\fraku}} 
    \setminus \{\bm{0}\} \\ k_j \in \{0,\ldots, b^{m- w_j}-1\}}} \widehat{\chi_{[\bm{0},\overline{\bsx}_{\fraku})}}(\bsk_{\fraku})
    \sum_{n=0}^{b^m-1}\wal_{\bsk_{\fraku}}(\bsz_{n,\fraku})} \\
    & \le \sum_{\substack{\bsk_{\fraku} \in \NN_0^{\abs{\fraku}} \setminus \{\bm{0}\} \\ k_j \in \{0,\ldots, b^{m- w_j}-1\}} } 
    \prod_{j\in\fraku}\rho(k_j)\abs{ \frac{1}{b^m}\sum_{n=0}^{b^m-1}\wal_{\bsk}(\bsz_{n,\fraku}) } \\
    &= \sum_{\substack{\bsk_{\setu} \in \NN_0^{\abs{\setu}} \setminus \{\bm{0}\} \\ k_j \in \{0,\ldots, b^{m- w_j}-1\} \\ \sum_{j\in\setu} (\widehat{C}^{(j)})^\top \Vec{k}_j \equiv \Vec{0} \bmod{b}}} \prod_{j\in\setu} \rho(k_j).
\end{align*}
This completes the proof.
\end{proof}

Let $\bsw=(w_j)_{j=1}^s$, $ 0= w_1 \le w_2 \le \cdots\le w_s $, and let $\widehat{P}_m:= P_m(\{\widehat{C}^{(j)}\}_j) = \{\bsz_0,\ldots,\bsz_{b^m-1}\} $ be generated by Algorithm \ref{alg:red_tsnet}. Let $\fraku\neq \emptyset$, $\fraku\subseteq [s]$ be given. We then define the reduced dual net, 
\begin{multline*}
P_{m,\fraku,\bsw}^{\perp}(\{\widehat{C}^{(j)}\}_j)\\ 
=\{\bsk_{\fraku} \in \NN_0^{\abs{\fraku}} : \, k_j \in \{0,\ldots, b^{m- \min(w_j,m)}  -1\} \,\forall j\in \fraku, 
\sum_{j\in \fraku} (\widehat{C}^{(j)})^\top \Vec{k}_j \equiv \Vec{0} \bmod{b}\},
\end{multline*}
and we also define the reduced dual net without zero components,
\begin{multline*}
P_{m,\fraku,\bsw}^{\perp,\ast}(\{ \widehat{C}^{(j)}\}_j)\\ 
=\{\bsk_{\fraku} \in \NN^{\abs{\fraku}} : \, k_j \in \{ 1,\ldots, b^{m- \min(w_j,m)}  -1\} \,\forall j\in \fraku, 
\sum_{j\in \fraku} (\widehat{C}^{(j)})^\top \Vec{k}_j \equiv \Vec{0} \bmod{b}\}.
\end{multline*}

Furthermore, we let
\begin{equation}
    R_{\bsw}(\{ \widehat{C}^{(j)} \}_{j\in \fraku}) 
    := \sum_{\bsk_{\fraku} \in P_{m,\fraku,\bsw}^{\perp}(\{\widehat{C}^{(j)} \}_j) \setminus \{\bm{0}\}} \prod_{j\in \fraku}\rho(k_j).
\end{equation}
Applying Proposition \ref{prop:reduced_discr} to all projections of $\widehat{P}_m$ onto the sets
$\fraku \subseteq [s]$, $\fraku\neq\emptyset$, gives the following bound on the weighted discrepancy.

\begin{proposition}
Let $m \in \NN$, let $\bsw$ be a given set of reduction indices with $0 = w_1 \le w_2 \le \cdots \le w_s$, let $s^\ast \in [s]$ be the largest index such that $w_{s^\ast} < m$,
and let $\widehat{P}_m:= P_m(\{\widehat{C}^{(j)}\}_j)$ be generated by Algorithm \ref{alg:red_tsnet}. Then, 
\begin{equation}\label{wgt_discr_bound}
 D_{b^m,\bsgamma}^{*}(\widehat{P}_m )
 \le \max_{\emptyset \neq \fraku \subseteq [s]} \gamma_{\fraku}  \begin{cases}
1 & \mbox{if } \fraku \not\subseteq [s^\ast], \\  \left[1 - \prod_{j\in\fraku} \left ( 1 - \frac{1}{b^{m-w_j}} \right )  
  + R_{\bsw}(\{\widehat{C}^{(j)}\}_{j\in \fraku}) \right] & \mbox{if } \fraku \subseteq [s^\ast]. \end{cases}
\end{equation}
\end{proposition}

We will now analyze the expressions occurring in the square brackets in \eqref{wgt_discr_bound} in greater detail. To this end, 
we restrict ourselves to product weights in the following, i.e., we assume weights $\gamma_{\fraku} = \prod_{j\in \fraku} \gamma_j$ with $\gamma_1 \ge \gamma_2\ge  \cdots >0$.

Then, using the second inequality of Lemma \ref{lem:diffprod} yields for the first term for the case $\fraku \subseteq [s^\ast]$ in \eqref{wgt_discr_bound},
\begin{equation}\label{wgt_discr_bound_part1}
\gamma_{\fraku} \left(1 - \prod_{j\in \setu} \left ( 1 - \frac{1}{b^{m-w_j}} \right ) \right) 
\le
\frac{1}{b^m}  \gamma_{\setu} \sum_{j\in \setu} b^{w_j}
\le 
\frac{1}{b^m}  \prod_{j\in \setu} \gamma_j (1 + b^{w_j}). 
\end{equation} 

For the case $\fraku \not\subseteq [s^\ast]$ in \eqref{wgt_discr_bound}, we use that $w_j\ge m$ if $j\in \setu\setminus [s^\ast]$, and obtain for $\setv = \setu \cap [s^\ast]$ that
\begin{equation}\label{wgt_discr_bound_part2}
 \gamma_{\setu} \le \gamma_{\setv} \gamma_{\setu\setminus\setv} \frac{1}{b^m}\prod_{j\in\setu\setminus\setv} (1+b^{w_j})
 \le \frac{1}{b^m}  \prod_{j\in \setu} \gamma_j (1 + b^{w_j}). 
\end{equation}

Regarding the remaining term in \eqref{wgt_discr_bound}, we show the following lemma.
\begin{lemma}
Let $\bsw$ be a given set of reduction indices with $0 = w_1 \le w_2 \le \cdots \le w_s$, 
and let $\widehat{P}_m:= P_m(\{\widehat{C}^{(j)}\}_j)$ be generated by Algorithm \ref{alg:red_tsnet}.
Assume that the matrices $\widehat{C}^{(1)}, \ldots, \widehat{C}^{(s)} \in \FF_b^{m \times m}$ are 
the generating matrices of a digital $((t_\fraku)_{\fraku \subseteq [s]},m,s)$-net. As above, 
let $s^\ast \in [s]$ be the largest number such that $w_{s^\ast} < m$, and assume that
$\frakv\neq\emptyset$, $\frakv\subseteq [s^*]$. Then,
\begin{eqnarray}\label{wgt_discr_bound_part3}
\lefteqn{R_{\bsw}(\{\widehat{C}^{(j)}\}_{j\in \frakv})}\nonumber\\ &\le& 
\sum_{\emptyset\neq\frakp\subseteq \frakv}
\frac{b^{t_{\frakp} }}{b^{m}} \left[\frac{1}{b} \left(\frac{b^2 + b}{3}\right)^{\abs{\frakp}}  \max\left( \frac{(m-t_\frakp)^{\abs{\frakp} -1}}{(\abs{\frakp} - 1)!}, \frac{1}{b} \right) + \left( \frac{b^2-1}{3 b} \right)^{\abs{\frakp}} \prod_{j \in \frakp} (m-w_j) \right].\nonumber\\
\end{eqnarray}
\end{lemma}

\begin{proof}
Recall that for each $\widehat{C}^{(j)}$, $j\in\frakv$, only the first $m-w_j$ rows of $\widehat{C}^{(j)}$ are non-zero. Consequently,
\begin{align}\label{eq:sum_frakp}
R_{\bsw}(\{{C}^{(j)}\}_{j\in \frakv}) 
= & \sum_{\bsk_{\frakv} \in P_{m,\frakv,\bsw}^{\perp}(\{\widehat{C}^{(j)}\}) \setminus \{\bm{0}\}} \ \ \prod_{j\in \frakv}\rho(k_j) \nonumber \\ 
\le &   \sum_{\bsk_{\frakv} \in P_{m,\frakv,\bszero}^{\perp}(\{\widehat{C}^{(j)}\}) \setminus \{\bm{0}\}} \ \ \prod_{j\in \frakv}\rho(k_j)\nonumber\\
= & \sum_{\emptyset\neq\frakp\subseteq \frakv} \ \ \
\sum_{ \bsk_{\frakp} \in P_{m,\frakp,\bszero}^{\perp,\ast}(\{\widehat{C}^{(j)}\}) }\ \ \ \prod_{j\in \frakp}\rho(k_j).
\end{align}
We use the estimate $(\sin(x))^{-1} \le (\sin(x))^{-2}$ for $0 < x < \pi$ to estimate $\rho(k_j) \le \frac{1}{b^r \sin^2(\pi \kappa_{j,a_j-1} /b)}$ for positive $k_j$, where we write $k_j = \kappa_{j,0} + \kappa_{j,1} b + \cdots + \kappa_{j, a_j-1} b^{a_j-1}$, with $\kappa_{j,a_j-1}\neq 0$. 
Now we can adapt the proof of \cite[Lemma~16.40]{DP10}, where we replace 
$\frac{1}{b^{2r}} \left(\frac{1}{\sin^2(\pi \kappa_{j, a_j-1} /b) } - \frac{1}{3} \right) $ with 
$\frac{1}{b^r \sin^2(\pi \kappa_{j, a_j-1}/b)}$, and $[s]$ by $\frakp$ to get the result. 

To simplify the notation we prove an upper bound on the inner sum in \eqref{eq:sum_frakp} 
for the special case $\frakp = [s^\ast]$ and assume that the underlying point set generated by 
$\{\widehat{C}_j\}_{j\in\frakp}=\{\widehat{C}_j\}_{j\in [s^\ast]}$ is a digital $(t,m,s^\ast)$-net. Then we obtain
\begin{align*}
& \sum_{ \bsk_{\frakp} \in P_{m,\frakp,\bszero}^{\perp,\ast}(\{\widehat{C}^{(j)}\}) }\ \ \ \prod_{j\in \frakp}\rho(k_j) \\ 
=& \sum_{ \bsk_{[s^*]} \in P_{m,[s^*],\bszero}^{\perp,\ast}(\{\widehat{C}^{(j)}\})}\ \ \ \prod_{j=1}^{s^*}\rho(k_j) \\
\le & \sum_{a_1=1}^{m-w_1} \cdots \sum_{a_{s^\ast}=1}^{m-w_{s^\ast}} b^{-a_1-\cdots - a_{s^\ast}} \underbrace{ \sum_{k_1=b^{a_1-1}}^{b^{a_1}-1} \cdots \sum_{k_{s^\ast}=b^{a_{s^\ast}-1}}^{b^{a_{s^\ast}}-1} }_{(\widehat{C}^{(1)})^\top \vec{k}_1 + \cdots + (\widehat{C}^{(s^\ast))^\top} \vec{k}_{s^\ast} \equiv  \vec{0} \pmod{b}} \prod_{j=1}^{s^\ast} \frac{1}{\sin^2 (\pi \kappa_{j,a_j-1} / b)} 
\\ 
\le & \sum_{a_1=1}^{m-w_1} \cdots \sum_{a_{s^\ast}=1}^{m-w_{s^\ast}} b^{-a_1-\cdots - a_{s^\ast}}  \left(\sum_{\kappa=1}^{b-1} \frac{1}{\sin^2(\pi \kappa/b)} \right)^{s^\ast} \times \\ 
& \begin{cases} 
0 & \mbox{if } a_1 + \cdots + a_{s^\ast} \le m-t, \\ 1 & \mbox{if } m-t < a_1 + \cdots + a_{s^\ast} \le m-t+s^\ast, \\ b^{a_1+\cdots + a_{s^\ast}- s^\ast -m+t} & \mbox{if } a_1+\cdots + a_{s^\ast} > m-t+s^\ast, 
\end{cases}
\end{align*}
where the second inequality follows from estimating the number of solutions of the linear system $(\widehat{C}^{(1)})^\top \vec{k}_1 + \cdots + (\widehat{C}^{(s^\ast)})^\top \vec{k}_{s^\ast} \equiv  \vec{0} \pmod{b}$, which was done in the proof of \cite[Lemma~16.40]{DP10}. From \cite[Corollary~A.23]{DP10} we have $\sum_{\kappa = 1}^{b-1} \frac{1}{\sin^2(\pi \kappa /b) } = \frac{b^2-1}{3}$.

Set
\begin{align*}
\Sigma_1 := & \left(\frac{b^2-1}{3}\right)^{s^\ast} \underbrace{ \sum_{a_1=1}^{m-w_1} \cdots \sum_{a_{s^\ast}=1}^{m-w_{s^\ast}} }_{m-t+1 \le a_1 + \cdots + a_{s^\ast} \le m-t+s^\ast } b^{-a_1-\cdots - a_{s^\ast}}, \\
\Sigma_2 := & \left(\frac{b^2-1}{3}\right)^{s^\ast} \underbrace{ \sum_{a_1=1}^{m-w_1} \cdots \sum_{a_{s^\ast}=1}^{m-w_{s^\ast}} }_{a_1 + \cdots + a_{s^\ast} > m-t+s^\ast } b^{-s^\ast - m + t},
\end{align*}
then the inner sum in \eqref{eq:sum_frakp} is bounded by
$ \Sigma_1 + \Sigma_2$.

If $m-t+1 - s^\ast \ge 0$, then we have
\begin{align*}
\Sigma_1 = & \left(\frac{b^2-1}{3 b}\right)^{s^\ast} \underbrace{ \sum_{b_1=0}^{m-w_1-1} \cdots \sum_{b_{s^\ast}=0}^{m-w_{s^\ast}-1} }_{m-t+1-s^\ast \le b_1 + \cdots + b_{s^\ast} \le m-t  } b^{-b_1-\cdots - b_{s^\ast}} \\ \le & \left(\frac{b^2-1}{3 b}\right)^{s^\ast}  \sum_{\ell = m-t - s^\ast+1}^{m-t } b^{-\ell} {\ell + s^\ast -1 \choose s^\ast-1} \\ \le &  \left(\frac{b^2-1}{3 b}\right)^{s^\ast}  \sum_{\ell = m-t - s^\ast+1}^{\infty} b^{-\ell} {\ell + s^\ast -1 \choose s^\ast-1} \\ \le &  \left(\frac{b^2-1}{3 b}\right)^{s^\ast} \frac{1}{b^{m-t-s^\ast+1}} {m-t \choose s^\ast-1}  \left( \frac{b}{b-1}  \right)^{s^\ast} \\ = & \left( \frac{b^2 + b}{3} \right)^{s^\ast} \frac{b^t}{b^{m+1}}\, \frac{(m-t)^{s^\ast-1}}{(s^\ast-1)!},
\end{align*}
where we used \cite[Lemma~13.24]{DP10} to estimate the infinite sum.

If $m-t+1 - s^\ast < 0$, then we have
\begin{align*}
\Sigma_1 \le & \left(\frac{b^2-1}{3 b}\right)^{s^\ast}  \sum_{\ell = 0}^\infty {\ell + s-1 \choose s-1} b^{-\ell}  \le \left(\frac{b^2-1}{3 b}\right)^{s^\ast} \left( \frac{b}{b-1}  \right)^{s^\ast} \le \left(\frac{b^2 + b}{3} \right)^{s^\ast} \frac{b^t}{b^m} \frac{1}{b^2}.
\end{align*}

For $\Sigma_2$ we use the estimate
\begin{equation*}
\Sigma_2 \le \left(\frac{b^2-1}{3 b } \right)^{s^\ast} \frac{b^t}{b^m} \sum_{a_1=1}^{m-w_1} \cdots \sum_{a_{s^\ast}=1}^{m-w_{s^\ast}} 1 \le  \left(\frac{b^2-1}{3b } \right)^{s^\ast} \frac{b^t}{b^m} \prod_{j \in [s^\ast]} (m-w_j).
\end{equation*}

The argument for the case $\frakp=[s^\ast]$ can be repeated analogously for all $\emptyset\neq\frakp\subseteq \frakv$ in \eqref{eq:sum_frakp}, 
by adapting notation, and in particular by replacing $t$ by $t_{\frakp}$. This yields the result claimed in the lemma, by plugging these estimates into \eqref{eq:sum_frakp}.
\end{proof}

Inserting the estimates in \eqref{wgt_discr_bound_part1}, \eqref{wgt_discr_bound_part2}, and \eqref{wgt_discr_bound_part3} into \eqref{wgt_discr_bound} yields 
the following theorem. 

\begin{theorem}\label{thm:advanced_discr_bound}
Let $\bsw$ be a given set of reduction indices with $0 = w_1 \le w_2 \le \cdots \le w_s$, 
and let $\widehat{P}_m:= P_m(\{\widehat{C}^{(j)}\}_j)$ be generated by Algorithm \ref{alg:red_tsnet}. 
Furthermore, assume product weights  $\gamma_{\fraku} = \prod_{j\in \fraku} \gamma_j$ with $\gamma_1 \ge \gamma_2\ge  \cdots >0$. Then, 
\begin{eqnarray}\label{adv_discr_bound}
 D_{b^m,\bsgamma}^{*}(\widehat{P}_m)
 &\le& \max_{\emptyset \neq \fraku \subseteq [s]} 
 \left[\frac{1 }{b^m}  \prod_{j\in \setu} \gamma_j (1 + b^{w_j})\right]\nonumber\\
 &&+  \max_{\emptyset \neq \frakv \subseteq [s^\ast] } 
 \left[\gamma_{\frakv} 
 \sum_{\emptyset\neq\frakp\subseteq \frakv}
\frac{b^{t_{\frakp} }}{b^{m}} \left[\frac{1}{b} \left(\frac{b^2 + b}{3}\right)^{\abs{\frakp}}  \max\left( \frac{(m- t_\frakp)^{\abs{\frakp} -1}}{(\abs{\frakp} - 1)!}, \frac{1}{b} \right)\right.\right.\nonumber\\ 
&& \left.\left.+ \left( \frac{b^2-1}{3 b} \right)^{\abs{\frakp}} \prod_{j \in \frakp} (m-w_j) \right]\right].
 \end{eqnarray}
\end{theorem}

We impose that the term
\[
 \max_{\emptyset \neq \fraku \subseteq [s]}  \left[\frac{1}{b^m}  \prod_{j\in \setu} \gamma_j (1 + b^{w_j})\right]
\]
in \eqref{adv_discr_bound} be bounded by $\kappa/b^m$ for some constant $\kappa> 0$ independent of $s$. Let $j_0\in \NN$ be minimal such that $\gamma_j \le 1$ for all $j > j_0$. Then we impose $ \prod_{j=1}^{s} \gamma_j(1 + b^{w_j})\le  \gamma_1^{j_0}\prod_{j=1}^{s} (1 + \gamma_j b^{w_j})\le \kappa $. Hence it is sufficient to choose $\kappa >  \gamma_1^{j_0}$ and  for all $j \in [s]$, 
\begin{equation}\label{wchoice}
w_j := \min\left(\floor{\log_b\left(\frac{\left(\frac{\kappa}{ \gamma_1^{j_0}}\right)^{1/s}-1}{\gamma_j}\right)},m\right).
\end{equation}

\begin{corollary}\label{cor:disc_bound}
Let $\bsgamma$  be product weights of the form  $\gamma_{\fraku} = \prod_{j\in \fraku} \gamma_j$ with $\gamma_1 \ge \gamma_2\ge  \cdots >0$ such that $\sum_{j=1}^\infty \gamma_j < \infty$. Let $C^{(1)}, \ldots, C^{(s)} \in \FF_b^{m \times m}$ be the generating matrices of a digital $((t_{\fraku})_{\fraku \subseteq [s]},m,s)$-net. Let the reduction indices $\bsw$ be chosen according to \eqref{wchoice} and let $s^\ast \in [s]$ be the largest number such that $w_{s^\ast} < m$.  Then there is a constant $C > 0$ independent of $s$ and $m$, such that
\begin{align*}
& D_{b^m,\bsgamma}^{*}(P_m(\{C^{(j)}\}_j)) \le \frac{C}{b^m} \\ 
& +  
\max_{\emptyset \neq \frakv \subseteq [s^\ast]} 
 \left[\gamma_{\frakv} 
 \sum_{\emptyset\neq\frakp\subseteq \frakv}
\frac{b^{t_{\frakp} }}{b^{m}} \left[\frac{1}{b} \left(\frac{b^2 + b}{3}\right)^{\abs{\frakp}}  \max\left( \frac{(m-t)^{\abs{\frakp} -1}}{(\abs{\frakp} - 1)!}, \frac{1}{b} \right)\right.\right.  \\ &
 \left.\left.+ \left( \frac{b^2-1}{3 b} \right)^{\abs{\frakp}} \prod_{j \in \frakp} (m-w_j) \right]\right].
\end{align*}

\end{corollary}

\begin{remark}
Note that the choice of the quantities $w_j$ in \eqref{wchoice} depends on $s$. For sufficiently fast decaying weights $\gamma_j$, it is possible to 
choose the $w_j$ such that they do no longer depend on $s$. Indeed, suppose, e.g., that $\gamma_j=j^{-2}$. Then we could choose the $w_j$ such that, 
for some $\tau\in(1,2)$,
\[
  w_j \le \min\left(\floor{\log_b\left( j^{2-\tau}\right)},m\right).
\]
This then yields
\begin{eqnarray*}
  \prod_{j=1}^s (1+\gamma_j b^{w_j})\le \exp\left(\sum_{j=1}^s \log (1+\gamma_j b^{w_j})\right)
  \le \exp\left(\sum_{j=1}^s \gamma_j b^{w_j}\right)\le \exp (\zeta(\tau)),
\end{eqnarray*}
where $\zeta(\cdot)$ is the Riemann zeta function. This then yields a dimension-independent bound on the term $\prod_{j=1}^s \gamma_j (1+b^{w_j})$ from above.
\end{remark}
\begin{remark}
 The term involving the maximum in the error bound of Corollary \ref{cor:disc_bound} crucially depends on the weights $\bsgamma$ and their interplay 
 with the $t$-values of the projections of $\widehat{P}_m$. In particular, small $t$-values in combination with sufficiently fast decaying weights should yield 
 tighter error bounds. However, the analysis of $t$-values of $(t,m,s)$-nets is in general non-trivial (see, e.g., \cite{DP10}). 
\end{remark}

\section{Reduced Monte Carlo}\label{sec:reduced_MC}

The idea of reduction is not limited to QMC algorithms, but can also be applied to Monte Carlo algorithms, as shall be discussed in this section.

\medskip

Let $N = b^m$ for some $b, m \in \mathbb{N}$ with $b \ge 2$. Further let $0 = w_1 \le w_2 \le w_3 \le \cdots \le w_s \le w_{s+1} =  m$ be some integers.
Let $N_j = N b^{-w_j} = b^{m-w_j}$ and $N_{s+1} = 1$. In particular, $N_1 = N$. Further we define $M_j = N_j/ N_{j+1} = b^{w_{j+1}-w_j}$ for
$j = 1, 2, \ldots, s-1$ and $M_s = N_s = b^{m-w_s}$. Then $N_j = M_j M_{j+1} \cdots M_{s-1} M_s$ and any integer $0 \le n < N_j$
can be represented by $n = m_s N_{s+1} + m_{s-1} N_{s} + \cdots + m_j N_{j+1}$ with $0 \le m_j < M_j$ for $1 \le j \le s$.

For each coordinate $1 \le j \le s$ we generate $N_j$ i.i.d. samples $y_{j,0}, \ldots, y_{j, N_j-1}$.
Different coordinates are also assumed to be independent.

Now for $0 \le m_j < M_j$ for $1 \le j \le s$ let
\begin{equation}\label{randpts}
x_{j, m_s N_{s+1} + m_{s-1} N_{s} + \cdots + m_1 N_2} = y_{j, m_s N_{s+1} + \cdots + m_j N_{j+1} }.
\end{equation}
This means that in coordinate $j$ we only have $N_j$ different i.i.d. samples.

\subsection{Computational cost reduction}\label{costred}

For each $0 \le n < N$ we need to compute
\begin{equation*}
x_{1,n}\bsa_1 + x_{2,n} \bsa_2 + \cdots + x_{s,n} \bsa_s,
\end{equation*}
where $\bsa_j$ is the $j$-th row of $A$. Using \eqref{randpts} we can write this as
\begin{equation*}
\sum_{j=1}^s y_{j, m_s N_{s+1} + \cdots + m_j N_{j+1}} \bsa_j,
\end{equation*}
which we need to compute for each $0 \le m_j < M_j$. We can do this recursively in the following way:
\begin{itemize}
\item First compute: $z_{s,m_s} = y_{s, m_s} \bsa_s$ for $0 \le m_s < M_s$ and store the results.
\item For $j = s-1, s-2, \ldots, 1$ compute:
\begin{equation*}
z_{j, m_s N_{s+1} + \cdots + m_j N_{j+1}} = y_{j, m_s N_{s+1} + \cdots + m_j N_{j+1}} \bsa_j + z_{j+1, m_s N_{s+1} + \cdots + m_{j+1} N_{j+2}} 
\end{equation*}
for $m_j = 0, 1, \ldots, M_j-1$, and store the resulting vectors.
\end{itemize}

Computing the values $z_{s,m_s}$ costs $\mathcal{O}(\tau N_s)$ operations. Computing $z_{j, m_s N_{s+1} + \cdots + m_j N_{j+1}}$ costs 
$\mathcal{O}\left(\tau N_j\right)$ operations.

Computing all the values therefore costs
\begin{equation*}
\mathcal{O}\left( \tau \left(N_s + N_{s-1} + \cdots + N_1 \right) \right) = \mathcal{O}\left( \tau b^m \left( b^{-w_1} + b^{-w_2} + \cdots + b^{-w_s} \right) \right)
\end{equation*}
operations.

If $\sum_{j=1}^\infty b^{-w_j} < \infty$, then the computational cost is independent of the dimension.

\subsection{Error analysis}

Since the samples are i.i.d., it follows that the estimator is unbiased, that is,
\begin{equation*}
\mathbb{E}(Q(f)) = \mathbb{E}(f).
\end{equation*}

For a given vector $\boldsymbol{x} = (x_1, \ldots, x_s)^\top$ and $\fraku\subseteq [s]$ let $\boldsymbol{x}_\fraku = (x_j)_{j \in \fraku}$ and $\boldsymbol{x}_{-\fraku} = (x_j)_{j \notin \fraku}$.
We now consider the variance of the estimator. Let
\begin{equation*}
\mu_u := \mathbb{E}_{\boldsymbol{x}_\fraku} \mathbb{E}_{\boldsymbol{x}_{-\fraku}} \mathbb{E}_{\boldsymbol{y}_{-\fraku}} f
= \int  \int f(\boldsymbol{x}_u\fraku, \boldsymbol{x}_{-\fraku}) \,\mathrm{d} \boldsymbol{x}_{-\fraku}
\int f(\boldsymbol{x}_\fraku, \boldsymbol{y}_{-\fraku}) \,\mathrm{d} \boldsymbol{y}_{-\fraku} \,\mathrm{d} \boldsymbol{x}_{\fraku}.
\end{equation*}
For instance, $\mu_{\emptyset} = (\mathbb{E}(f))^2$ and $\mu_{\{1, \ldots, s\}} = \int f^2$.

In classical Monte Carlo integration, one studies the variance $\mathrm{Var}(Q(f)) = (\mu_{\{1, \ldots, s\}} - \mu_\emptyset) / N$.
We now show how the reduced MC construction influences the variance.

\begin{theorem}
The variance of the reduced Monte Carlo estimator is given by
\begin{equation*}
\mathrm{Var}(Q(f)) = \sum_{k=1}^s \mu_{\{k, k+1, \ldots, s\}} \prod_{j=k}^s M_j^{-1} \left(1 - \frac{1}{M_{k-1}} \right) - \frac{\mu_{\emptyset}}{M_s},
\end{equation*}
where we set $1-\frac{1}{M_0} = 1$.
\end{theorem}

\begin{proof}
The variance of $Q(f)$ can be written as $\mathbb{E}(Q^2(f)) - (\mathbb{E}(Q(f)))^2$.
The last term $(\mathbb{E}(Q(f)))^2$ equals 
$(\mathbb{E}(f))^2 = \mu_{\emptyset}$.

We have
\begin{equation*}
Q(f) = \frac{1}{M_1} \sum_{m_1 = 0}^{M_1-1} \cdots \frac{1}{M_s} \sum_{m_s =0 }^{M_s-1} f(y_{1, n_1}, \ldots, y_{s, n_s}),
\end{equation*}
where $n_j = m_s N_{s+1} + \cdots + m_j N_{j+1}$. Hence
\begin{align*}
Q^2(f) = & \frac{1}{M_1^2} \sum_{m_1, m'_1 = 0}^{M_1-1} \cdots \frac{1}{M_s^2}
\sum_{m_s, m'_s =0 }^{M_s-1} f(y_{1, n_1}, \ldots, y_{s, n_s}) f(y_{1, n'_1}, \ldots, y_{s, n'_s}) \\
= & \sum_{u \subseteq \{1, \ldots, s\} }  \frac{1}{N^2} \sum f(y_{1,n_1}, \ldots, y_{s, n_s}) f(y_{1,n'_1}, \ldots, y_{s, n'_s}),
\end{align*}
where the second sum is over all $0 \le m_j, m'_j < M_j$ such that $m_j = m'_j$ for $j \in u$ and $m_j \neq m'_j$ for $j \notin u$.

Let $1 \le k \le s$. If $m_j = m'_j$ for $k \le j \le s$, then $n_j = n'_j$ for $k \le j \le s$ and if $m_{k-1} \neq m'_{k-1}$,
then $n_i \neq n'_i$ for $1 \le i < k$. In this case
\begin{equation*}
\mathbb{E}\left( f(y_{1,n_1}, \ldots, y_{s, n_s}) f(y_{1,n'_1}, \ldots, y_{s, n'_s}) \right) = \mu_{\{k, k+1, \ldots, s \}}.
\end{equation*}
The number of such instances is given by $\prod_{j = k}^s M_j (M_{k-1}^2 - M_{k-1}) \prod_{j=1}^{k-2} M_j^2$.
Since $N = M_1 M_2 \cdots M_s$, we obtain
$\prod_{j = k}^s M_j (M_{k-1}^2 - M_{k-1}) \prod_{j=1}^{k-2} M_j^2 N^{-2} = \prod_{j=k}^s M_j^{-1} (1-M_{k-1}^{-1})$.

If $m_s \neq m'_s$ we obtain that
\begin{equation*}
\mathbb{E}\left( f(y_{1,n_1}, \ldots, y_{s, n_s}) f(y_{1,n'_1}, \ldots, y_{s, n'_s}) \right) = \mu_{\emptyset}.
\end{equation*}
This case occurs $(M_s^2 - M_s) \prod_{j=1}^{s-1} M_j^2$ times, and therefore   $(M_s^2 - M_s) \prod_{j=1}^{s-1} M_j^2 N^{-2} = 1 - M_s^{-1}$.

Using the linearity of expectation we obtain the formula.
\end{proof}

\section{Numerical experiments}\label{sec:num_exp}
In this section we give exemplary numerical results regarding the use of reduced rank-1 lattice point sets for matrix products, as outlined in Section \ref{sec:fast_reduced_comp}.

\subsection{Reduced matrix-vector products}
In each case, we compute the generating vectors $ \bsz = (z_1b^{w_1},\ldots,z_s b^{w_s}) $ depending on the reduction indices $ w_j $ via a reduced CBC construction with product weights $ \gamma_j = 0.7^j $, as developed in \cite{DKLP15}.
For a fair comparison, we do not include in the timings the construction of $ \bsz $ and we average the computing times over 10 runs. Computations are run using MATLAB 2019a on an Octa-Core (Intel(R) Core(TM) i7-10510U CPU @ 1.80GHz) laptop.

As a first example, we illustrate the benefit of Algorithm \ref{alg:fast-mv-prod} compared to the standard matrix-vector product to compute $ P = X A $ for $ A \in \R^{s\times \tau} $.
In Figure \ref{fig:varyalg} we compare different combinations of $ s,m $, for the choice of reduction indices $ w_j = \min (\floor{\log_2 (j)},m) $ and fixed $ b = 2 $. 
We repeat the same experiment on Algorithm \ref{alg:fast-mv-prod-alt} with the same settings.

In Figures \ref{fig:varyalg}--\ref{fig:varyalgw}, the blue graphs show the results for the reduced matrix-matrix product according to Algorithm~\ref{alg:fast-mv-prod}, 
the red graphs show the results for the optimized reduced matrix-matrix product according to Algorithm~\ref{alg:fast-mv-prod-alt}, and the light brown graphs 
show the results for a straightforward implementation of the matrix-matrix product without any adjustments.

We conclude that the computational saving due to Algorithms \ref{alg:fast-mv-prod} and  \ref{alg:fast-mv-prod-alt} is more pronounced for larger $ m,s $.  Note that the right plot is in semi-logarithmic scale.

\begin{figure}[H]
    \includegraphics[clip,trim=2cm 8.5cm 3cm 8.5cm, width=0.5\linewidth]{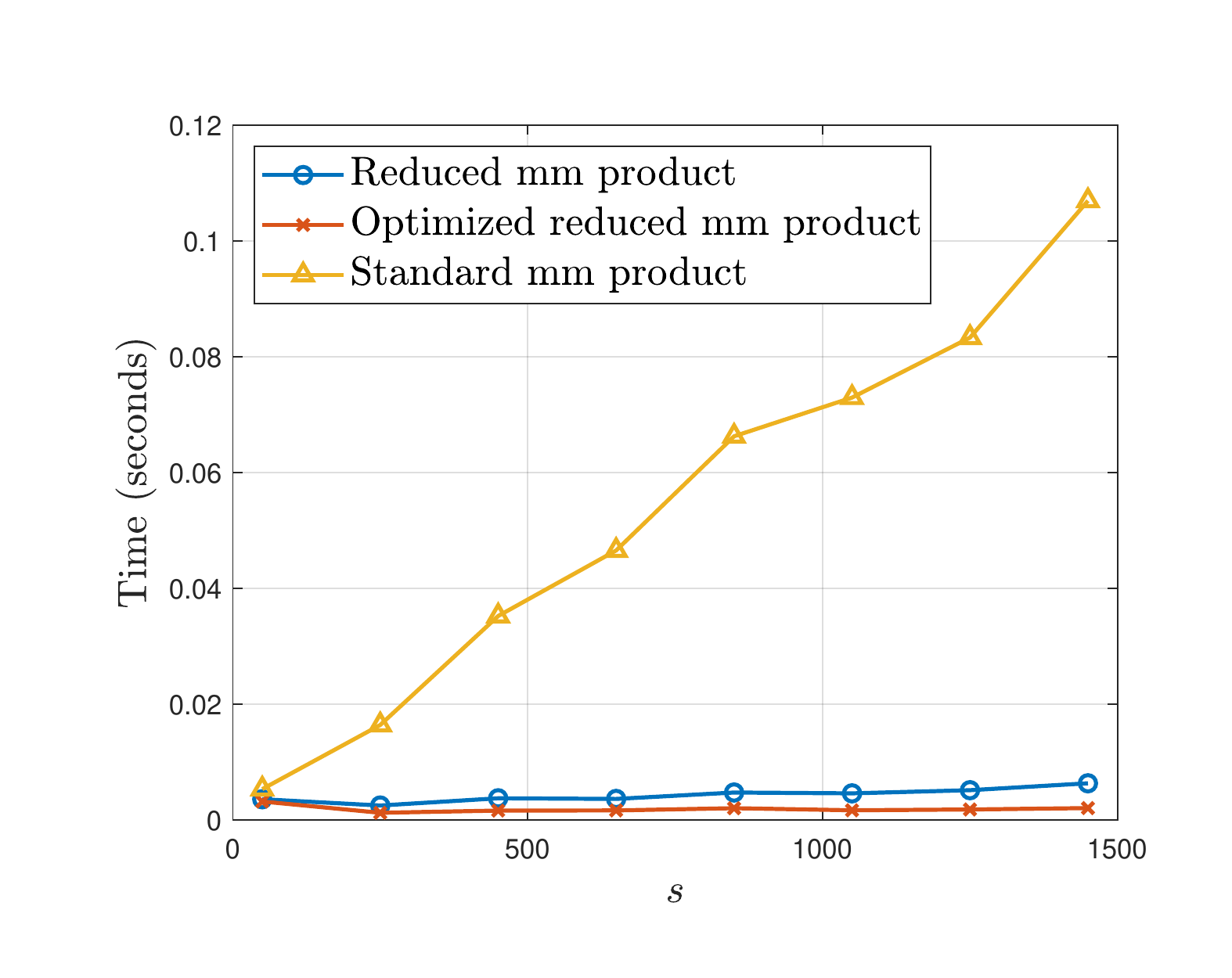}
    \includegraphics[clip,trim=2cm 8.5cm 3cm 8.5cm, width=0.5\linewidth]{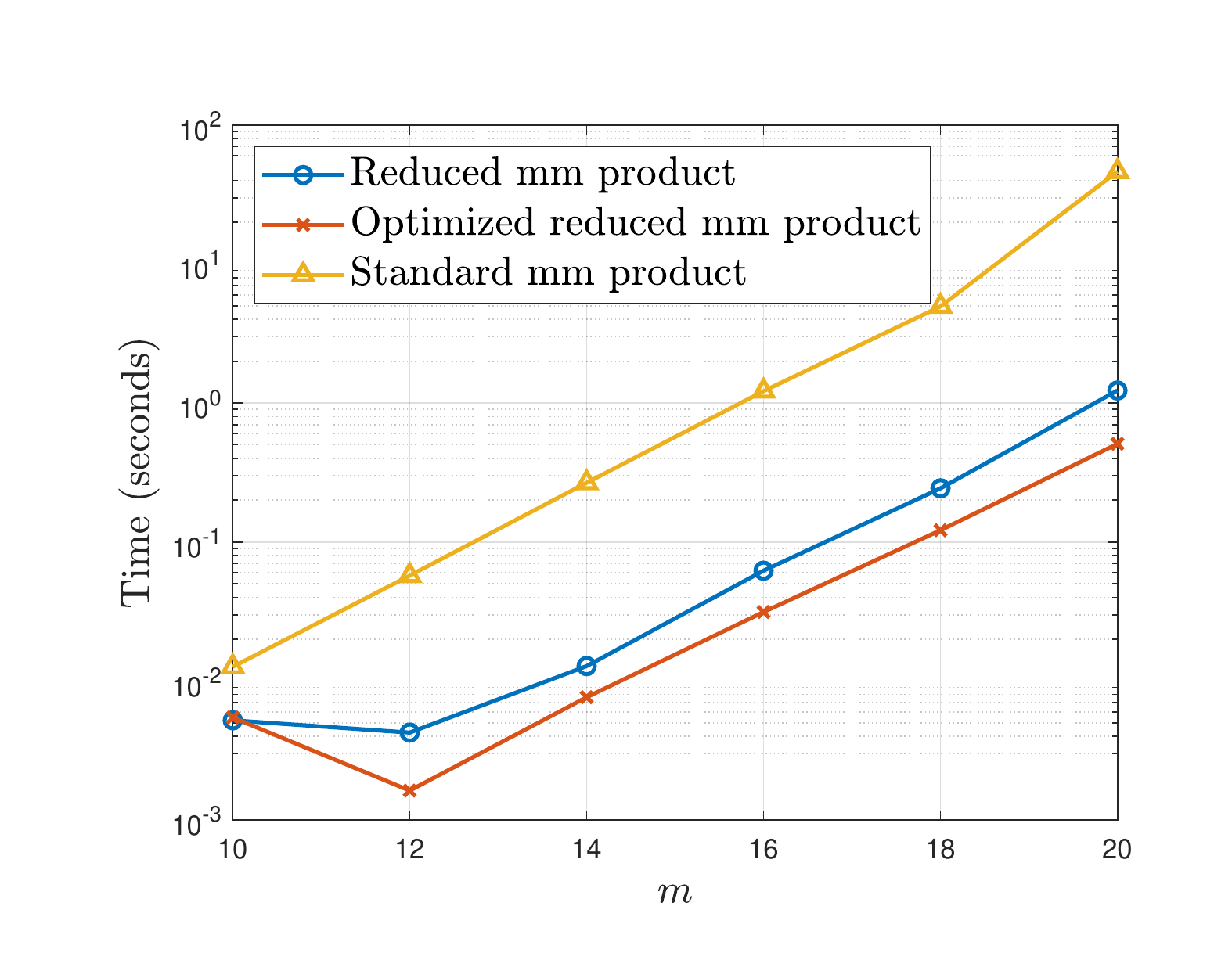}
    \caption{ $ m = 12 $, $ \tau =20 $, varying $ s $ (left) and $ s = 800 $, $ \tau = 20 $, varying $ m $ (right).}\label{fig:varyalg}
\end{figure}

Next we study in Figure \ref{fig:varyalgt} the behavior as the size $ \tau$ increases. Also here, we see a clear advantage
of Algorithms~\ref{alg:fast-mv-prod} and \ref{alg:fast-mv-prod-alt} over a straightforward implementation of the matrix-matrix product.

\begin{figure}[H]
    \includegraphics[clip,trim=2cm 8.5cm 3cm 8.5cm, width=0.5\linewidth]{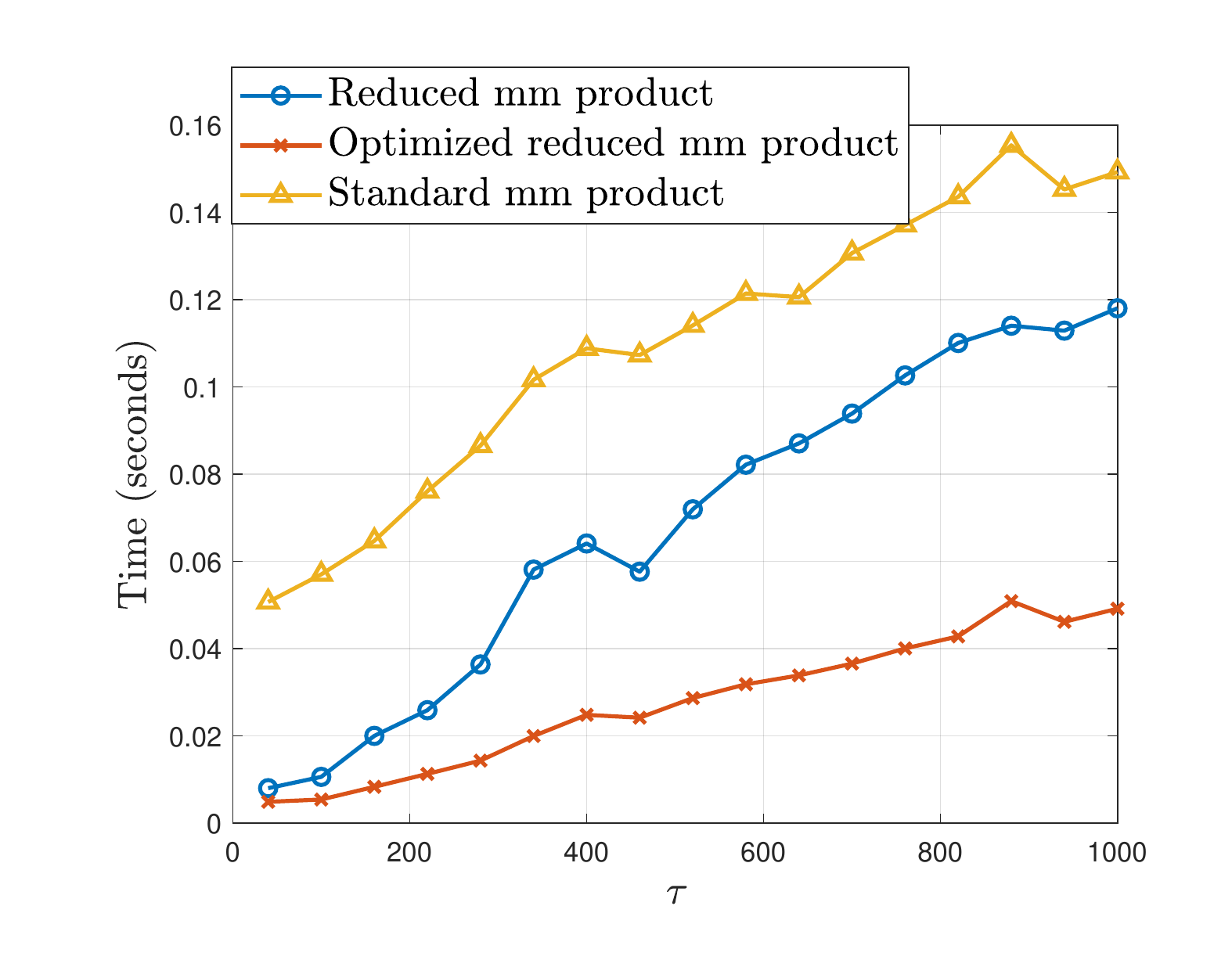}
    \includegraphics[clip,trim=2cm 8.5cm 3cm 8.5cm, width=0.5\linewidth]{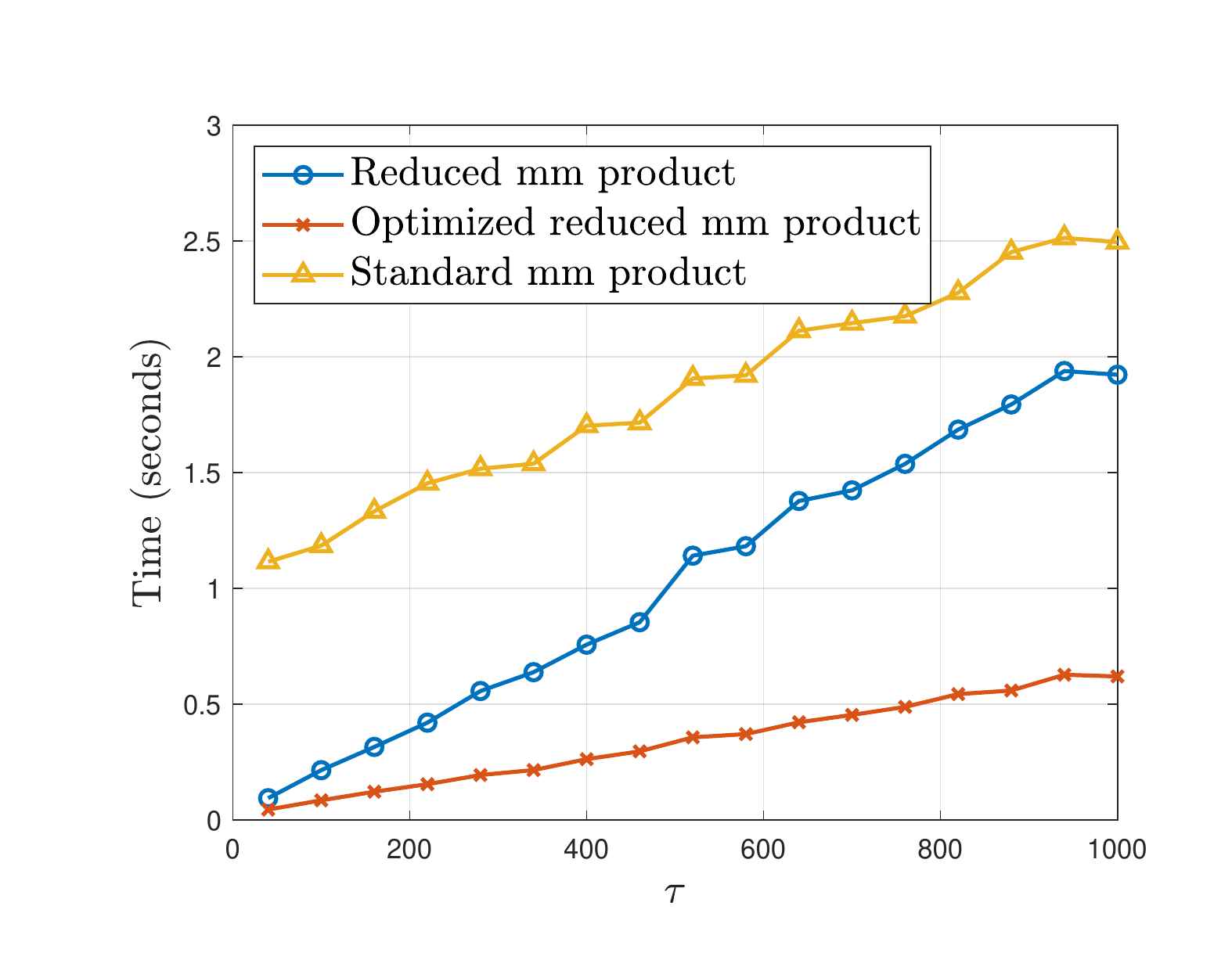}
    \caption{ $ s=800 $, varying $ \tau $ for $ m = 12 $ (left) and $m=16$ (right).}\label{fig:varyalgt}
\end{figure}

When the reduction is less aggressive, that is, $w_j$ increases more slowly, the benefit is still considerable for large $s$ especially for Algorithm \ref{alg:fast-mv-prod-alt}, see Figure \ref{fig:varyalgw}. 
\begin{figure}[H]
    \includegraphics[clip,trim=2cm 8.5cm 3cm 8.5cm, width=0.5\linewidth]{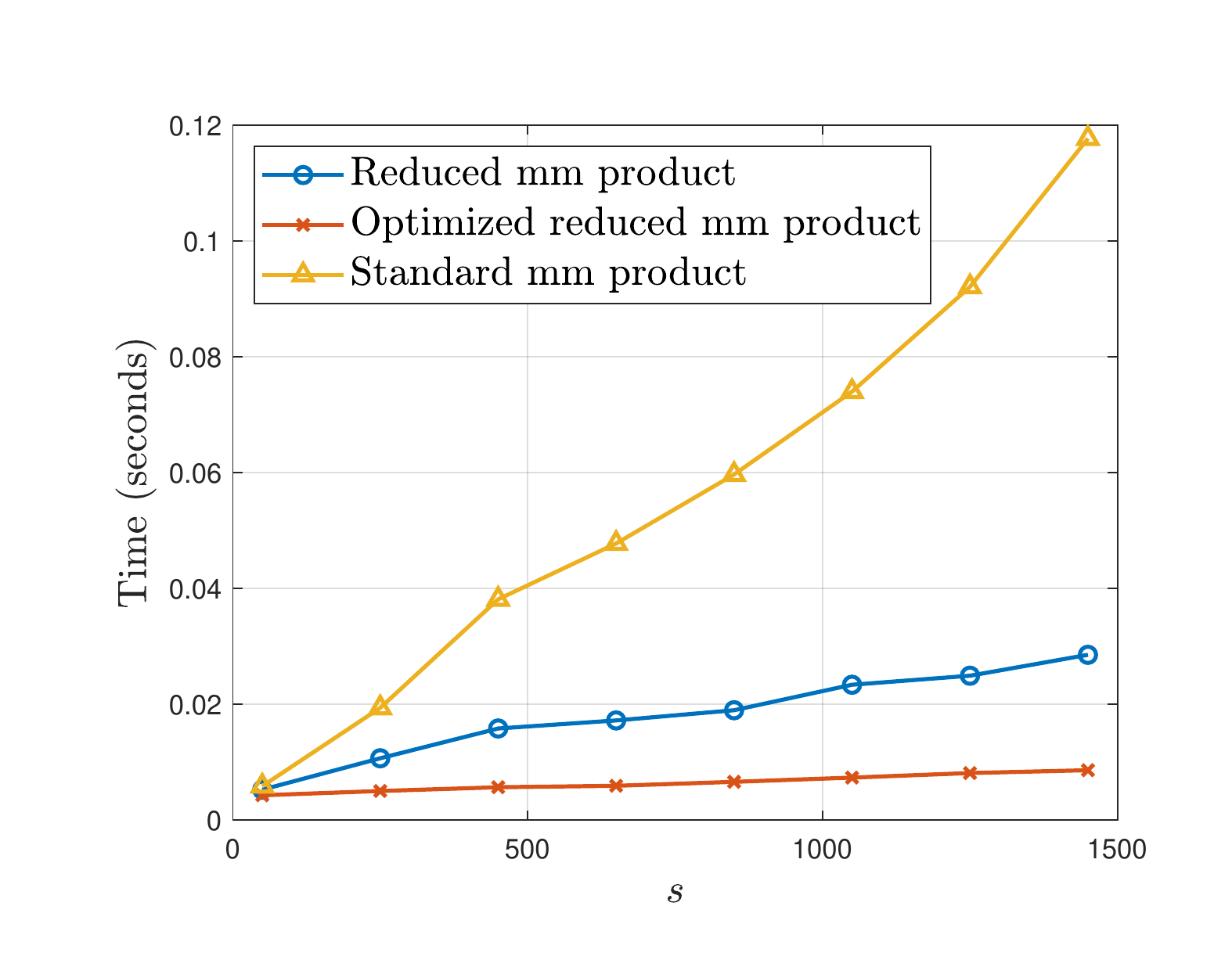}
    \includegraphics[clip,trim=2cm 8.5cm 3cm 8.5cm, width=0.5\linewidth]{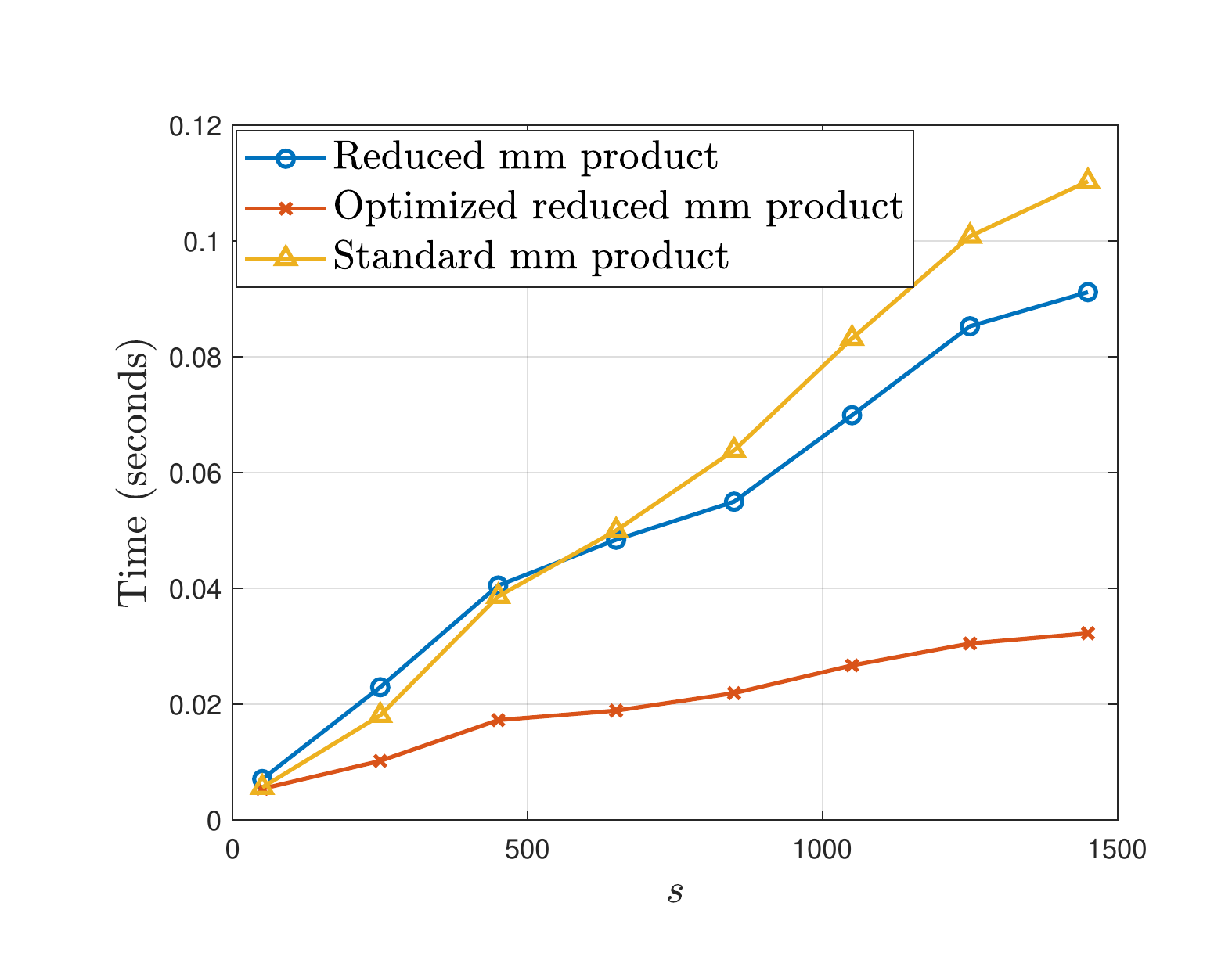}

    \caption{ $ m = 12 $, $ \tau = 20 $, varying $ s $ for $ w_j = \min\left(\floor{\log_2 (j^{1/2})},m\right) $ (left) and $ w_j = \min\left(\floor{\log_2 (j^{1/4})},m\right) $ (right).}\label{fig:varyalgw}
\end{figure}

We now test the reduced matrix-vector product for Monte Carlo integration with respect to the normal distribution. 
As an example, we consider the pricing of a basket option \cite[Section 3.2.3]{G04}.
We define the payoff $ H(S) = \max(\frac1s\sum_{j=1}^s S_j(T) - K, 0 )$, 
where $S_j(T)$ is the price of the $j$-th asset at maturity $T$. 
Under the Black and Scholes model with zero interest rate we have $S_j(T) = S_j(0)\exp(-\Sigma_{jj}/2T + W_j\sqrt{T}) $, where $W_j = LZ $, $ Z \sim \mathcal{N}(0,{\rm Id}_{s\times s})$ and $LL^T = \Sigma$ is the covariance matrix of the random vector $S = (S_1,\ldots,S_s)$.

We set $S_j(0) = 100$ for all $j$, $s = 10, T=1$, strike price $ K=110$, and as the covariance matrix we pick $\sigma = 0.4, \rho= 0.2$ and
\begin{equation*}
    \Sigma = \begin{pmatrix}
        \sigma &\rho& & & \\
        \rho & \sigma &\rho& & \\
        &\ddots&\ddots&\ddots & \\
        &&\rho&\sigma & \rho \\
        & && \rho&\sigma  \\
    \end{pmatrix}.
\end{equation*}
We approximate the option price $\mathbb{E}(S) \approx \frac{1}{b^m}\sum_{k=1}^{b^m} S_j(0)\exp(-\sigma/2T + L\bsx_k\sqrt{T})$, with $\bsx_k$ random samples from $Z$. The main work is to compute $XL^{\top}$ (recall that $X$ was defined in \eqref{QMC_matrix}) and thus the reduced matrix-vector multiplication can be beneficial in this example.
Results for different choices of reduction indices $w_j$ are displayed in Figure \ref{fig:optionw}, where we plot the mean error over $R=5$ repetitions for different values of reduction indices, using $Rb^m, m = 25$, Monte Carlo samples for the reference value. Note that the performance of QMC methods in this illustration appears to 
be not particularly strong as compared to standard Monte Carlo, as we consider a setting without coordinate weights, which usually is unfavorable for QMC methods.

\begin{figure}[H]
    \includegraphics[clip,trim=2cm 8.5cm 3cm 8.5cm, width=0.5\linewidth]{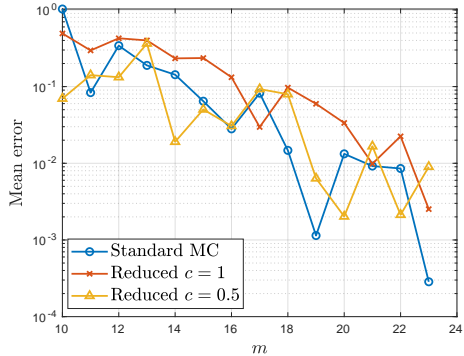}
    \caption{ Option pricing example: standard Monte Carlo (corresponding to $c=0$, compared with reduced Monte Carlo for $ w_j = \min\left(\floor{\log_2 (j^{c})},m\right) $, $c\in \{ 1, 0.5\}$.  }\label{fig:optionw}
\end{figure}  

\section*{Acknowledgements}

Josef Dick is supported by the Australian Research Council Discovery Project DP220101811. Adrian Ebert and Peter Kritzer acknowledge the support of the Austrian Science Fund (FWF) Project F5506, which is part of the Special Research Program ``Quasi-Monte Carlo Methods: Theory and Applications''. Furthermore, Peter Kritzer has partially been supported by the Austrian Science Fund (FWF) Project P34808. For the purpose of open access, the authors have applied a CC BY public copyright licence to any author accepted manuscript version arising from this submission.

\begin{small}
	\noindent\textbf{Authors' addresses:}\\

        \noindent Josef Dick\\
	School of Mathematics and Statistics\\
	University of New South Wales (UNSW)\\
	Sydney, NSW, 2052, Australia\\
	\texttt{josef.dick@unsw.edu.au}\\

        \noindent Adrian Ebert\\
                Johann Radon Institute for Computational and Applied Mathematics (RICAM)\\
	Austrian Academy of Sciences\\
	Altenbergerstr. 69, 4040 Linz, Austria\\
    and\\
    Centrica Business Solutions\\
    Roderveldlaan 2, 2600 Antwerpen, Belgium
	\texttt{adrian.ebert@hotmail.com}\\

        \noindent Lukas Herrmann\\
        Johann Radon Institute for Computational and Applied Mathematics (RICAM)\\
	Austrian Academy of Sciences\\
	Altenbergerstr. 69, 4040 Linz, Austria\\
	\texttt{lukas.herrmann@alumni.ethz.ch}\\
 
	\noindent Peter Kritzer\\
	Johann Radon Institute for Computational and Applied Mathematics (RICAM)\\
	Austrian Academy of Sciences\\
	Altenbergerstr. 69, 4040 Linz, Austria\\
	\texttt{peter.kritzer@oeaw.ac.at}\\
	
	\noindent Marcello Longo\\
	Seminar for Applied Mathematics\\
    ETH Z\"urich\\
    R\"amistrasse 101, 8092 Z\"urich, Switzerland\\
    \texttt{marcello.longo@sam.math.ethz.ch}

\end{small}

\end{document}